\documentclass[11pt]{article}%
\usepackage{amssymb}
\usepackage{amsmath}
\usepackage{amsfonts}
\usepackage{graphicx}
\usepackage{color}
\usepackage[top=2cm,bottom=2cm,left=2.5cm,right=2.5cm]{geometry}%
\providecommand{\U}[1]{\protect\rule{.1in}{.1in}}

\newtheorem{theorem}{Theorem}[section]

\newtheorem{proposition}{Proposition}[section]

\newtheorem{corollary}[theorem]{Corollary}

\newtheorem{definition}[theorem]{Definition}

\newtheorem{remark}[theorem]{Remark}

\newenvironment{proof}[1][Proof]{\textbf{#1.} }{\ \rule{0.5em}{0.5em}}
\begin{document}

\title{A priori estimates for solutions to anisotropic elliptic \\problems via symmetrization}
\author{A. Alberico\thanks{Istituto per le Applicazioni del Calcolo \textquotedblleft
M. Picone\textquotedblright, Sez. Napoli, C.N.R., Via P. Castellino 111 -
80131 Napoli, Italy. E--mail: a.alberico@na.iac.cnr.it} - G. di
Blasio\thanks{Dipartimento di Matematica e Fisica, Seconda Universit\`{a}
degli Studi di Napoli, Via Vivaldi, 43 - 81100 Caserta, Italia. E--mail:
giuseppina.diblasio@unina2.it} -- F. Feo\thanks{Dipartimento di Ingegneria,
Universit\`{a} degli Studi di Napoli \textquotedblleft
Pathenope\textquotedblright, Centro Direzionale Isola C4 80143 Napoli, Italia.
E--mail: filomena.feo@uniparthenope.it}}
\maketitle

\begin{abstract}
{\scriptsize {\negthinspace\negthinspace\negthinspace\ } We obtain a
comparison result for solutions to nonlinear fully anisotropic elliptic
problems by means of anisotropic symmetrization. As consequence we deduce a
priori estimates for norms of the relevant solutions.}

\end{abstract}



\numberwithin{equation}{section}

\numberwithin{equation}{section}

\section{\textbf{Introduction}}

In the present paper we treat the following class of anisotropic problems
\begin{equation}
\left\{
\begin{array}
[c]{ll}%
-\operatorname{div}\left(  a\left(  x,u,\nabla u\right)  \right)  =f\left(
x\right)  -\operatorname{div}\left(  g\left(  x\right)  \right)  &
\mbox{ in }\Omega\\
& \\
u=0 & \mbox{ on }\partial\Omega,
\end{array}
\right.  \label{Prob'}%
\end{equation}
where $\Omega$ is a bounded open subset of $\mathbb{R}^{N},$ $N\geq2,$
$a:\Omega\times\mathbb{R}\times\mathbb{R}^{N}\rightarrow\mathbb{R}^{N}$ is a
Carath\'{e}odory function such that, for a.e. $x\in\Omega$,%
\begin{equation}
a(x,\eta,\xi)\cdot\xi\geq\Phi\left(  \xi\right)  \quad\hbox{\rm
for}\;(\eta,\xi)\in\ \mathbb{R}\times{{\mathbb{R}}^{N}} \label{elliti'}%
\end{equation}
with $\Phi$ an $N-$dimensional Young function, and the data $f$ and $g$ are
measurable functions fulfilling a suitable summability condition. We emphasize
that the anisotropy of problem (\ref{Prob'}) is governed by a general
$N-$dimensional convex function of the gradient not necessary of polynomial type.

In the classical theory of regularity for solutions to anisotropic problems in
partial differential equations and in calculus of variations the anisotropy
condition depends on differential operators whose growth with respect to the
partial derivatives of $u$ is governed by different powers (see \emph{e.g.}
\cite{Gi, Mar,mercaldo, BMS, Ho, FS, St2, ELM, FGK, FGL, antontsev-chipot-08, castro}).
Problems governed by fully anisotropic growth conditions as in (\ref{elliti'})
have been recently studied in \cite{Clocal, cianchi immersione, cianchi
anisotropo,AC,A}. \newline There is also a large number of papers related to a
different type of anisotropy (see \emph{e.g.} \cite{AFTL, BFK,DPG, DPdB}).

Our aim is to obtain a comparison result for solutions to problem
\eqref{Prob'} relying upon anisotropic symmetrization in the spirit of
\cite{cianchi anisotropo} . We prove that the symmetric rearrangement of a
solution to anisotropic problem \eqref{Prob'} is pointwise dominated by the
radial solution to an appropriate isotropic problem. The novelty of our paper
is the presence of datum in divergence form.

In order to give an idea of our results we restrict ourself to a class of
anisotropic problems having a growth in the different partial derivatives
controlled by different powers, \textit{i.e.} when in \eqref{elliti'} we take
$\Phi(\xi)=\sum_{i=1}^{N}\lambda_{i}|\xi_{i}|^{p_{i}}$ for $\xi\in
{{\mathbb{R}}^{N}}$. More precisely, we deal with the following class of
problems whose prototypal example is:
\begin{equation}
\ \left\{
\begin{array}
[c]{ll}%
-\underset{i=1}{\overset{N}{%
{\displaystyle\sum}
}}\left(  \left\vert u_{x_{i}}\right\vert ^{p_{i}-2}u_{x_{i}}\right)  _{x_{i}%
}=f\left(  x\right)  -\underset{i=1}{\overset{N}{%
{\displaystyle\sum}
}}\dfrac{\partial g_{i}}{\partial x_{i}} & \mbox{  in }\Omega\vspace{0.2cm}\\
& \\
u=0 & \mbox{ on }\partial\Omega,
\end{array}
\right.  \label{P_p_i}%
\end{equation}
where $p_{i}>1$ for $i=1,\ldots,N$.

As well known the time evolution versions of problem (\ref{P_p_i}) provide the
mathematical models for natural phenomena in biology and fluid mechanics. For
example, they are the mathematical description of the dynamics of fluids in
anisotropic media when the conductivities of the media are different in
different directions (see \textit{e.g.} \cite{ADS}). They also appear in
biology as a model for the propagation of epidemic diseases in heterogeneous
domains (see \textit{e.g.} \cite{BK}).

As typically happens using symmetrization methods (see \textit{e.g.}
\cite{Ta1} and the bibliography starting with it) we reduce the estimates of
solutions to problem \eqref{P_p_i} to some simpler one-dimensional
inequalities. Namely, for some positive constant $C$, we show that
\begin{equation}
u^{\bigstar}\left(  x\right)  \!\leq C \,v\left(  x\right)  \text{
\ \ \ \ \ for \textit{a.e.} \ }\Omega^{\bigstar}\!\!\vspace{0.2cm},
\label{comparison}%
\end{equation}
where $\Omega^{\bigstar}$ is the ball centered in the origin having the same
measure as $\Omega$, $u^{\bigstar}$ is the symmetric rearrangement of a
solution $\!u$ to problem (\ref{P_p_i}) and $v$ is the radial solution to the
following isotropic problem
\begin{equation}
\!\!\left\{  \!\!%
\begin{array}
[c]{ll}%
\!\!-\operatorname{div}\left(  \Lambda\left\vert \nabla v\right\vert
^{\overline{p}-2}\nabla v\right)  \!=\!\!f^{\bigstar}\left(  x\right)
\!-\!\operatorname{div}\left(  \!\left(  \overline{p}\Lambda\right)
^{\!1/\overline{p}\!}\!\left(  \!\overline{p}^{\prime}G\left(  \!\omega
_{N}\left\vert x\right\vert ^{N}\!\right)  \!\right)  ^{\!1/\overline
{p}^{\prime}\!}\!\dfrac{x}{\left\vert x\right\vert }\!\right)  \! &
\mbox{  in  }\Omega^{\bigstar}\!\!\vspace{0.2cm}\\
& \\
\!\!v=0 & \!\mbox{ on }\partial\Omega^{\bigstar}\!.\!\!
\end{array}
\!\!\right.  \label{Prob simm pi}%
\end{equation}
Here, $\overline{p}$ is the harmonic mean of exponents $p_{1},\dots,p_{N}$,
$\omega_{N}$ is the measure of the $N-$dimensional unit ball, $f^{\bigstar}$
is the symmetric decreasing rearrangement of $f$ and $G$ is a suitable
function related to $\sum_{i=1}^{N}g_{i}^{p_{i}^{\prime}}$. In contrast to the
isotropic case not only the domain and the data of problem (\ref{Prob simm pi}%
) are symmetrized, but also the ellipticity condition is subject to an
appropriate symmetrization. Indeed the operator in problem (\ref{Prob simm pi}%
) is the isotropic $\overline{p}-$Laplacian.

Estimate (\ref{comparison}) is the starting point to obtain bounds for norms
of the solutions to problem \eqref{P_p_i} in terms of norms of the data.
Indeed, in view of \eqref{comparison} and the explicit expression of $v$, any
estimate for a rearrangement invariant norm of $u$ is reduced to an
appropriate one-dimensional norm inequality.

The paper is organized as follows. In Section 2 we recall some background. In
Section 3 we present the main comparison theorem and some regularity results
concerning problem (\ref{Prob'}). Section 4 contains applications to two
special models of the $N-$dimension Young function $\Phi$. In particular, when
$\Phi(\xi)=\sum_{i=1}^{N}\lambda_{i}|\xi_{i}|^{p_{i}}$ for $\xi\in
{{\mathbb{R}}^{N}}$, it is possible to give the functional setting for problem
\eqref{P_p_i} and then \textit{a priori} estimate (\ref{comparison}) allows us
to state conditions on data for the existence of solutions and to complete the
framework of regularity, showing how the regularity of the data influences
solutions's one. \bigskip

\section{Preliminaries}

\subsection{Young function}

Let $\Phi:\mathbb{R}^{N}\rightarrow\left[  0,+\infty\right[  $ be a
$N-$\textit{dimensional Young function}, namely an even convex function such
that%
\[
\Phi\left(  0\right)  =0\text{ \ \ and }\underset{\left\vert \xi\right\vert
\rightarrow+\infty}{\lim}\Phi\left(  \xi\right)  \text{ }=+\infty.
\]
A standard case considered in literature is when%
\begin{equation}
\Phi\left(  \xi\right)  =\underset{i=1}{\overset{N}{%
{\displaystyle\sum}
}}\lambda_{i}\left\vert \xi_{i}\right\vert ^{p_{i}}\text{ \ for }\xi
\in\mathbb{R}^{N}, \label{Fi pi}%
\end{equation}
for some $\lambda_{i}>0$ and $p_{i}>1$ for any $i=1,\ldots,N.$ An extension of
$\left(  \ref{Fi pi}\right)  $ is given by
\begin{equation}
\label{ex2}\Phi\left(  \xi\right)  =\underset{i=1}{\overset{N}{%
{\displaystyle\sum}
}}\Upsilon_{i}(\xi_{i})\text{ \ for }\xi\in\mathbb{R}^{N},
\end{equation}
where $\Upsilon_{i}$\ for $i=1,\ldots, N$ are one-dimensional Young functions
vanishing only at zero. For instance we can choose
\begin{equation}
\label{ex3}\Upsilon_{i}(s)=\left\vert s\right\vert ^{p_{i}}\left(  \log\left(
c+\left\vert s\right\vert \right)  \right)  ^{\alpha_{i}}\text{ \ \ for }s\in%
\mathbb{R}%
\end{equation}
and for $i=1,\ldots,N,$ where either $p_{i}>1$ and $\alpha_{i}\in%
\mathbb{R}
$ or $p_{i}=1$ and $\alpha_{i}\geq0$, and constant $c$ is large enough for
$\Upsilon_{i}$ to be convex. \newline In the case $N=2$, an example of
function $\Phi$ which is neither radial nor sum of one-dimensional functions
is given by
\begin{equation}
\label{ex4}\Phi\left(  \xi\right)  =\left\vert \xi_{1}-\xi_{2}\right\vert
^{\alpha}+\left\vert \xi_{1}\right\vert ^{\beta}\left(  \log\left(
c+\left\vert \xi_{1}\right\vert \right)  \right)  ^{\delta}\text{ \ for
}\left(  \xi_{1},\xi_{2}\right)  \in%
\mathbb{R}
^{2},
\end{equation}
where $\alpha,\beta\geq1,$ $\delta$ is either any real number or a nonnegative
number according to whether $\beta>1$ or $\beta=1$ and constant $c$ is large
enough for $\Phi$ to be convex.

The Young inequality tells us that
\begin{equation}
\xi\cdot\xi^{\prime}\leq\Phi\left(  \xi\right)  +\Phi_{\bullet}\left(
\xi^{\prime}\right)  \text{ \ \ for \ }\xi,\xi^{\prime}\in\mathbb{R}^{N},
\label{Young}%
\end{equation}
where $\Phi_{\bullet}$ is the \textit{Young conjugate} function of $\Phi$
given by%
\begin{equation}
\Phi_{\bullet}\left(  \xi^{\prime}\right)  =\sup\left\{  \xi\cdot\xi^{\prime
}-\Phi\left(  \xi\right)  :\xi\in\mathbb{R}^{N}\right\}  \text{ \ \ for \ }%
\xi^{\prime}\in\mathbb{R}^{N}. \label{Young func}%
\end{equation}
Here, \textquotedblleft$\;\cdot\;$\textquotedblright\ stands for scalar
product in ${{\mathbb{R}}^{N}}$. We observe that the function $\Phi_{\bullet}$
enjoys the same properties as $\Phi$ and is a $N$-dimensional Young function
if
\begin{equation}
\underset{\left\vert \xi\right\vert \rightarrow+\infty}{\lim}\frac{\Phi\left(
\xi\right)  }{\left\vert \xi\right\vert }=+\infty. \label{lim}%
\end{equation}
We recall that if $\digamma:\mathbb{R}^{N}\rightarrow\left[  0,+\infty\right[
$ is a convex function such that $\digamma\left(  0\right)  =0,$ then
\begin{equation}
\digamma\left(  \lambda\xi\right)  \left\{
\begin{array}
[c]{lll}%
\leq\lambda\digamma\left(  \xi\right)  & \text{if} & 0<\lambda\leq1\\
\geq\lambda\digamma\left(  \xi\right)  & \text{if} & \lambda\geq1
\end{array}
\right.  \label{costante}%
\end{equation}
for $\xi\in\mathbb{R}^{N}.$ Moreover, if $h:\left[  0,+\infty\right[
\rightarrow\left[  0,+\infty\right[  $ is convex and $h(0)=0$, then
\begin{equation}
h(s_{1}+s_{2})\geq h(s_{1})+h(s_{2})\text{ \ }\forall s_{1},s_{2}\in\left[
0,+\infty\right[  . \label{sublineare}%
\end{equation}

\subsection{Symmetrization}

A precise statement of our result requires the use of classical notions of
rearrangement and pseudo-rearrangement of a function, and of suitable
symmetrization of a Young function, introduced by Klimov in \cite{Klimov 74}.
\newline Let $u$ be a measurable function (continued by $0$ outside its
domain) fulfilling
\begin{equation}
\left\vert \{x\in\mathbb{R}^{N}:\left\vert u(x)\right\vert >t\}\right\vert
<+\infty\text{ \ \ \ for every }t>0. \label{insieme livello di misura finita}%
\end{equation}
The \textit{symmetric decreasing rearrangement} of $u$ is the function
$u^{\bigstar}:\mathbb{R}^{N}\rightarrow\left[  0,+\infty\right[  $
$\ $satisfying
\begin{equation}
\{x\in\mathbb{R}^{N}:u^{\bigstar}(x)>t\}=\{x\in\mathbb{R}^{N}:\left\vert
u(x)\right\vert >t\}^{\bigstar}\text{ \ for }t>0. \label{livello palla}%
\end{equation}
The \textit{decreasing rearrangement} $u^{\ast}$ of $u$ is defined as
\[
u^{\ast}(s)=\sup\{t>0:\mu_{u}(t)>s\}\text{ \ for }s\geq0,
\]
where
\[
\mu_{u}(t)=\left\vert \{x\in{\Omega}:\left\vert u(x)\right\vert
>t\}\right\vert \text{ \ \ \ \ for }t\geq0
\]
denotes the \textit{distribution function} of $u$. \newline Moreover
\[
u^{\bigstar}(x)=u^{\ast}(\omega_{N}\left\vert x\right\vert ^{N})\text{
\ \ }\hbox{\rm for a.e.}\;x\in{{\mathbb{R}}^{N}.}%
\]
Analogously, we define the \textit{symmetric increasing rearrangement}
$u_{\bigstar}$ on replacing \textquotedblleft$>$\textquotedblright\ by
\textquotedblleft$<$\textquotedblright\ in the definitions of the sets in
(\ref{insieme livello di misura finita}) and (\ref{livello palla}).

In what follows, we shall also make use of the function $u^{\ast\ast}$ defined
as
\[
u^{\ast\ast}(s)=\frac{1}{s}\int_{0}^{s}u^{\ast}(r)\;dr\qquad
\hbox{\rm for}\;s>0.
\]
Note that $u^{\ast\ast}$ is also non-increasing, and satisfies $u^{\ast
}(s)\leq u^{\ast\ast}(s)$ for $s>0$. We refer to \cite{BS} for details on
these topics.

Now, let us introduce the notion of pseudo-rearrangement (see \textit{e.g.}
\cite{AT}). Let $u$ be a measurable function on $\Omega$ and $h\in
L^{1}(\Omega)$. We will say that a function $H(s)$ is the
\textit{pseudo-rearrangement} of $h$ with respect to $u$ if there exists a
family $\mathcal{E(}\Omega\mathcal{)}$ of measurable subsets $E(s)$ of
$\Omega$\ with measure $s\in(0,\left\vert \Omega\right\vert )$ such that if
$s_{1}\leq s_{2}$ then $E(s_{1})\subseteq E(s_{2}),$ $E(s)=\left\{  x\in
\Omega:\left\vert u(x)\right\vert >u^{\ast}(s)\right\}  $ if there exists
$t\in%
\mathbb{R}
$ such that $s=\mu(t)$ and%

\[
H(s)=\frac{d}{ds}\int_{E(s)}h(x)dx\text{\ \ \ for a.e. \ }s\in(0,\left\vert
\Omega\right\vert ).
\]
Recall that if $h\in L^{p,q}(\Omega),$ then $H\in L^{p,q}(0,|\Omega|)$ and the
following estimate holds
\begin{equation}
\Vert H\Vert_{L^{p,q}(0,|\Omega|)}\leq\Vert h\Vert_{L^{p,q}(\Omega)},
\label{norm_pseudo}%
\end{equation}
where $L^{p,q}(\Omega)$ denotes the usual Lorentz spaces with $1<p\leq+\infty$
and $0<q\leq+\infty$.

Let $\Phi$ be a $N-$dimensional Young function. We denote by $\Phi
_{\blacklozenge}:\mathbb{R\rightarrow}\left[  0,+\infty\right[  $ the
symmetrization of $\Phi$ introduced in \cite{Klimov 74}. It is the
one-dimensional Young function fulfilling%
\begin{equation}
\Phi_{\blacklozenge}(\left\vert \xi\right\vert )=\Phi_{\bullet\bigstar\bullet
}\left(  \xi\right)  \text{ \ for }\xi\in\mathbb{R}^{N}, \label{def fi rombo}%
\end{equation}
namely it is the composition of Young conjugation, symmetric increasing
rearrangement and Young conjugate again. We stress that the functions
$\Phi_{\blacklozenge}$ and $\Phi_{\bigstar}$ are not equals in general, but
they are always equivalent, \emph{i.e.} there exist two positive constants
$K_{1}$ and $K_{2}$ such that
\begin{equation}
\Phi_{\bigstar}(K_{1}\xi)\leq\Phi_{\blacklozenge}(|\xi|)\leq\Phi_{\bigstar
}(K_{2}\xi)\qquad\hbox{\rm for}\;\;\xi\in{{\mathbb{R}}^{N}}. \label{equiv}%
\end{equation}
Moreover $\Phi_{\blacklozenge}(\left\vert \,\cdot\,\right\vert )=\Phi
_{\bigstar}\left(  \,\cdot\,\right)  $ if and only if $\Phi$ is
radial,\textit{ i.e.} $\Phi=\Phi_{\bigstar}.$ \newline When $\Phi$ is given by
\eqref{Fi pi}, easy calculations show (see \textit{e.g. }\cite{cianchi
anisotropo}), that
\begin{equation}
\Phi_{\blacklozenge}(\left\vert \xi\right\vert )=\Lambda\left\vert
\xi\right\vert ^{\overline{p}}, \label{fi rombo}%
\end{equation}
where $\overline{p}$ is the harmonic mean of exponents $p_{1},\ldots,p_{N}$,
\textit{i.e.}%
\begin{equation}
\frac{1}{\overline{p}}=\frac{1}{N}\overset{N}{\underset{i=1}{\sum}}\frac
{1}{p_{i}} \label{p barrato bis}%
\end{equation}
and
\begin{equation}
\Lambda=\frac{2^{\overline{p}}\left(  \overline{p}-1\right)  ^{\overline{p}%
-1}}{\overline{p}^{\overline{p}}}\left[  \frac{\underset{i=1}{\overset{N}{\Pi
}}p_{i}^{\frac{1}{p_{i}}}\left(  p_{i}^{\prime}\right)  ^{\frac{1}%
{p_{i}^{\prime}}}\Gamma(1+1/p_{i}^{\prime})}{\omega_{N}\Gamma(1+N/\overline
{p}^{\prime})}\right]  ^{\frac{\overline{p}}{N}}\left(  \underset
{i=1}{\overset{N}{\Pi}}\lambda_{i}^{\frac{1}{p_{i}}}\right)  ^{\frac
{\overline{p}}{N}} \label{lapda}%
\end{equation}
with $\omega_{N}$ the measure of the $N-$dimensional unit ball, $\Gamma$ the
Gamma function and $p_{i}^{\prime}=\frac{p_{i}}{p_{i}-1}$ with the usual
conventions if $p_{i}=1$. \newline In the more general case \eqref{ex2}, it is
possible to show (see \textit{e.g.} \cite{cianchi immersione}), that
\begin{equation}
\Phi_{\blacklozenge}^{-1}(r)\approx\left(  \underset{i=1}{\overset{N}{\Pi}%
}\Upsilon_{i}^{-1}(.)\right)  ^{\frac{1}{N}}(r)\text{ \ for }r\geq0,
\label{inversa fi rombo}%
\end{equation}
where $\Upsilon_{i}^{-1}$ denotes the inverse function of $\Upsilon_{i}$ to
$\left[  0,+\infty\right)  $ (see \textit{e.g.} \cite{cianchi immersione}).
From (\ref{inversa fi rombo}) we deduce that%
\begin{equation}
\Phi_{\blacklozenge}(s)\approx\left\vert s\right\vert ^{\overline{p}}\left(
\log\left(  c+\left\vert s\right\vert \right)  \right)  ^{\frac{\overline{p}%
}{N}\underset{i=1}{\overset{N}{%
{\textstyle\sum}
}}\frac{\alpha_{i}}{p_{i}}} \label{fi rombo con log}%
\end{equation}
near infinity, when $\Upsilon_{i}$ is given by \eqref{ex3}. \newline Finally,
when $\Phi$ is defined as in \eqref{ex4}, $\Phi_{\blacklozenge}(s)$ is
equivalent to $\left\vert s\right\vert ^{\frac{2\alpha\beta}{\alpha+\beta}%
}\log^{\frac{\alpha\delta}{\alpha+\beta}}\left(  c+\left\vert s\right\vert
\right)  $ near infinity and to $\left\vert s\right\vert ^{\frac{2\alpha\beta
}{\alpha+\beta}}$ near zero.

In what follows we will use the function
\begin{equation}
\Psi_{\blacklozenge}(s)=\left\{
\begin{array}
[c]{ll}%
\displaystyle\frac{\Phi_{\blacklozenge}(s)}{s} & \hbox{\rm for $s >0 $}\\
& \\
0 & \hbox{\rm for $s=0$}
\end{array}
\right.  \label{Psi}%
\end{equation}
that is not-decreasing. If, in addiction,%
\begin{equation}
\underset{s\rightarrow0^{+}}{\lim}\frac{\Phi_{\blacklozenge}\left(  s\right)
}{s}=0, \label{psi incrising}%
\end{equation}
then $\Psi_{\blacklozenge}$ $\ $is strictly increasing in $\left[
s_{0},+\infty\right)  $ with
\begin{equation}
s_{0}=\sup\left\{  s\geq0:\Phi_{\blacklozenge}\left(  s\right)  =0\right\}  .
\label{s0}%
\end{equation}
Finally we observe that%
\begin{equation}
\Phi_{\blacklozenge\bullet}\left(  r\right)  \leq\Phi_{\blacklozenge}\left(
\Psi_{\blacklozenge}^{-1}\left(  r\right)  \right)  \text{ \ for \ \ }r\geq0,
\label{dis PsiRombo}%
\end{equation}
where $\Psi_{\blacklozenge}^{-1}$ denotes the inverse of $\Psi_{\blacklozenge
}$ restricted to $[s_{0},+\infty)$.

We remember that in the anisotropic setting a \textit{Polya-Szeg\"{o}
principle} holds (see \cite{cianchi anisotropo}). Let $u$ be a weakly
differentiable function in $%
\mathbb{R}
^{N}$ satisfying (\ref{insieme livello di misura finita}) and such that
$\int_{\mathbb{R}^{N}}\Phi\left(  \nabla u\right)  dx<+\infty,$ then
$u^{\bigstar}$ is weakly differentiable in $%
\mathbb{R}
^{N}$ and
\begin{equation}
\int_{\mathbb{R}^{N}}\Phi_{\blacklozenge}\left(  \left\vert \nabla
u^{\bigstar}\right\vert \right)  dx\leq\int_{\mathbb{R}^{N}}\Phi\left(  \nabla
u\right)  dx\text{ .\ } \label{cianchi}%
\end{equation}

\subsection{Function spaces}

Finally we recall some definitions on function spaces that we will be useful
in the following. In our framework, where any kind of Young function is
admitted, Orlicz norm are a natural class in which obtain a priori estimate.
Given a one-dimensional Young function $A,$ the \textit{Orlicz space}
$L^{A}\left(  \Omega\right)  $ is defined as the Banach space of all
measurable functions $h$ on $\Omega$ such that the Luxemburg norm
%
\begin{equation}
\Vert h\Vert_{L^{A}(\Omega)}=\inf\left\{  k>0:\int_{\Omega}A\left(
\frac{h(x)}{k}\right)  \;dx\leq1\right\}  \label{L^A_norm}%
\end{equation}
is finite. We observe that $\Vert h\Vert_{L^{A}(\Omega)}$ is a rearrangement
invariant (briefly \textit{r.i.}) norm since $\Vert h\Vert_{L^{A}(\Omega
)}=\Vert h^{\ast}\Vert_{L^{A}(0,|\Omega|)}$ (see \textit{e.g.} \cite{BS}). As
well-known examples of Orlicz spaces are Lebesgue spaces and Lorentz spaces.

\noindent Now we introduce the Orlicz-Lorentz space $X_{A,N}(\Omega).$

Let $a$ be the non-decreasing, left-continuous function in $[0,\infty)$ such
that
\[
A(s)=\int_{0}^{|s|}a(r)\;dr\qquad\hbox{\rm for}\;s\in\mathbb{R}.
\]
We define $B$ as
\[
B(s)=\int_{0}^{|s|}b(r)\,dr\qquad\text{for}\;s\in\mathbb{R},
\]
where $b$ is the non-decreasing, left-continuous function in $[0,\infty)$
whose (generalized) left-continuous inverse obeys
\begin{equation}
b^{-1}(s)=\left(  \int_{a^{-1}(s)}^{\infty}\,\left(  \int_{0}^{r}\left(
\frac{1}{a(\tau)}\right)  ^{\frac{1}{N-1}}\,d\tau\right)  ^{-N}\,\frac
{dr}{a(r)^{N^{\prime}}}\right)  ^{\frac{1}{1-N}}\qquad\text{for}\;s\geq0,
\label{b}%
\end{equation}
and $a^{-1}$ and $b^{-1}$ are the (generalized) left-continuous inverse of $a$
and $b$, respectively. We define the \textit{Orlicz-Lorentz space}
$X_{A,N}(\Omega)$ as the set of the measurable functions $h$ in $\Omega$ such
that the quantity
\begin{equation}
\Vert h\Vert_{X_{A,N}(\Omega)}=\Vert s^{-\frac{1}{N}}h^{\ast}(s)\Vert
_{L^{B}(0,|\Omega|)} \label{O-L_norm}%
\end{equation}
is finite (see \textit{e.g. }\cite{cianchi immersione}). We observe that
$X_{A,N}(\Omega)$ is a \textit{r.i.} space endowed with the norm
\eqref{O-L_norm}. We may assume, without loss of generality, that the integral
appearing in relation \eqref{b} converges, namely
\[
\int_{0}\left(  \frac{1}{a(s)}\right)  ^{\frac{1}{N-1}}\,ds<+\infty,
\]
or, equivalently, that
\begin{equation}
\int_{0}\left(  \frac{s}{A(s)}\right)  ^{\frac{1}{N-1}}\,ds<+\infty.
\label{int_A}%
\end{equation}
Indeed such a condition is not a restriction, since replacing $A$, if
necessary, by an equivalent Young function near $0$ and making the integral in
\eqref{int_A} converge, turns $\Vert\cdot\Vert_{L^{A}(\Omega)}$ into an
equivalent norm.

\noindent We conclude this section by recalling the definition of
\textit{Lorentz-Zygmund space} $L^{p,q}\left(  \log L\right)  ^{\alpha}\left(
\Omega\right)  $ for $0<p,q\leq\infty$, $-\infty<\alpha<+\infty$. It is the
space of all measurable functions $h$ on $\Omega$ such that%
\begin{equation}
||h||_{L^{p,q}(\log L)^{\alpha}(\Omega)}=\left\{
\begin{array}
[c]{ll}%
\left(  {\displaystyle\int_{0}^{\left\vert \Omega\right\vert }}\left[
s^{\frac{1}{p}}\left(  1+\log\frac{\left\vert \Omega\right\vert }{s}\right)
^{\alpha}h^{\ast}(s)\right]  ^{q}\frac{ds}{s}\right)  ^{\frac{1}{q}} &
\text{if }0<q<\infty\\
\underset{s\in(0,\left\vert \Omega\right\vert )}{\sup}\left[  s^{\frac{1}{p}%
}\left(  1+\log\frac{\left\vert \Omega\right\vert }{s}\right)  ^{\alpha
}h^{\ast}(s)\right]  & \text{if\ }q=\infty
\end{array}
\right.  \label{definizione Lorenz-Zygmund}%
\end{equation}
is finite. Note that if $p=+\infty$, in order to obtain a non trivial space,
we have to require $q<+\infty$ and $\alpha+\frac{1}{q}<0$ or $q=+\infty$ and
$\alpha\leq0$.

\section{Assumptions and main results}

Let us consider the following problem%
\begin{equation}
\left\{
\begin{array}
[c]{ll}%
-\operatorname{div}\left(  a\left(  x,u,\nabla u\right)  \right)  =f\left(
x\right)  -\operatorname{div}\left(  g\left(  x\right)  \right)  &
\mbox{ in }\Omega\\
& \\
u=0 & \mbox{ on }\partial\Omega,
\end{array}
\right.  \label{Prob}%
\end{equation}
where $\Omega$ is an bounded open subset of $\mathbb{R}^{N},$ $N\geq2,$
$a:\Omega\times\mathbb{R}\times\mathbb{R}^{N}\rightarrow\mathbb{R}^{N}$ is a
Carath\'{e}odory function such that, for a.e. $x\in\Omega$,%
\begin{equation}
a(x,\eta,\xi)\cdot\xi\geq\Phi\left(  \xi\right)  \text{ } \label{elliti}%
\end{equation}
for $\left(  \eta,\xi\right)  \in\mathbb{R}\times\mathbb{R}^{N},\Phi
:\mathbb{R}^{N}\rightarrow\left[  0,+\infty\right[  $ is $N-$dimensional Young function.

\noindent Moreover we make the following assumption on the data%
\begin{equation}
s^{\frac{1}{N}}f^{\ast\ast}(s)\in L^{\Phi_{\blacklozenge\bullet}}\left(
0,\left\vert \Omega\right\vert \right)  \text{ and }\int_{\Omega}\Phi
_{\bullet}\left(  g\right)  dx<\infty, \label{dati}%
\end{equation}
where $\Phi_{\blacklozenge}$ is defined in (\ref{def fi rombo}) and
$\Phi_{\bullet}$ is defined in (\ref{Young func})$.$ A function $u\in
V_{0}^{1,\Phi}(\Omega)$ is a weak solution to (\ref{Prob}) if%
\begin{equation}
\int_{\Omega}a\left(  x,u,\nabla u\right)  \cdot\nabla\varphi dx=\int_{\Omega
}\left(  f\varphi+g\cdot\nabla\varphi\right)  dx\text{ }\ \ \text{for every
}\varphi\in V_{0}^{1,\Phi}(\Omega), \label{sol deb}%
\end{equation}
where
\[
V_{0}^{1,\Phi}(\Omega)=\left\{
\begin{array}
[c]{c}%
u:u\text{ is real-valued function in }\Omega\text{ whose continuation by
}0\text{ outside }\Omega\text{ is weakly }\\
\text{differentiable in }\mathbb{R}^{N}\text{ and satisfy }\int_{\Omega}%
\Phi\left(  \nabla u\right)  dx<\infty
\end{array}
\right\}  .
\]
We remark that\ $V_{0}^{1,\Phi}(\Omega)$ is always a convex set, but not
necessary a linear space (unless $\Phi$ satisfies the so called $\Delta_{2}%
-$condition, see \cite{rao}). We stress that we are not interested in the
existence of solution to problem (\ref{Prob}). Nevertheless using embedding
theorem for anisotropic Orlicz-Sobolev spaces (see \cite{cianchi immersione})
we expound some conditions to give on $a(x,\eta,\xi),f\left(  x\right)  $ and
$g\left(  x\right)  $ in order to guarantee that (\ref{sol deb}) is well
posed. More precisely if
\begin{equation}
\int^{+\infty}\left(  \frac{s}{\Phi_{\blacklozenge}(s)}\right)  ^{\frac
{1}{N-1}}<+\infty, \label{existence_1}%
\end{equation}
then\ any function $u\in V_{0}^{1,\Phi}(\Omega)$ is essentially bounded.
Consequently the left-hand side of (\ref{sol deb}) is always finite and
right-hand side of (\ref{sol deb}) is finite if $f\in$ $L^{1}\left(
\Omega\right)  $ and $\int_{\Omega}\Phi_{\bullet}\left(  g\right)  dx<\infty$.
Otherwise if
\begin{equation}
\int^{+\infty}\left(  \frac{s}{\Phi_{\blacklozenge}(s)}\right)  ^{\frac
{1}{N-1}}=+\infty, \label{existence_2}%
\end{equation}
then any function $u\in V_{0}^{1,\Phi}(\Omega)$ satisfies%
\[
\int_{\Omega}\Phi_{N}(c\,u(x))dx<+\infty
\]
for every $c\in%
\mathbb{R}
,$ where
\begin{equation}
\Phi_{N}(s)=\Phi_{\blacklozenge}(H_{\Phi_{\blacklozenge}}^{-1}(|s|))\qquad
\text{\ and \ }\qquad H_{\Phi_{\blacklozenge}}(r)=\left(  \int_{0}^{r}\left(
\frac{s}{\Phi_{\blacklozenge}(s)}\right)  ^{\frac{1}{N-1}}ds\right)
^{\frac{1}{N^{\prime}}} \label{H}%
\end{equation}
for $s\in%
\mathbb{R}
$ and $r\geq0.$ Obviously, for $H$ and $\Phi_{N}$ to be well defined,
$\Phi_{\blacklozenge}$ has to fulfill the following condition
\begin{equation}
\int_{0}\left(  \frac{s}{\Phi_{\blacklozenge}(s)}\right)  ^{\frac{1}{N-1}%
}\;ds<\infty. \label{cond1}%
\end{equation}
In order to guarantee (\ref{sol deb}) is well-posed we need to require, when
condition (\ref{existence_2}) is in force, not only (\ref{dati}) but also the
following one:
\[
\Phi_{\bullet}(a(x,\eta,\xi))\leq c\left[  \theta\left(  x\right)  +M\left(
\eta\right)  +\Phi(\xi)\right]  \qquad\text{for}\;\;a.e.\;\;x\in
\Omega,\;\text{for}\;\left(  \eta,\xi\right)  \in\mathbb{R}\times
\mathbb{R}^{N},
\]
where $c$ is a positive constant, $\theta\in L^{1}\left(  \Omega\right)  $ is
a positive function and $M:\mathbb{R\rightarrow}\left[  0,+\infty\right[  $ is
a continuous function such that $M\left(  \eta\right)  \leq\Phi_{N}\left(
k\eta\right)  $ for some $k>0$ and for every $\eta\in\mathbb{R}.$

For more comments on the consistence of equation (\ref{sol deb}) we refer to
\cite{cianchi anisotropo} and to Section 4, where a particular class of
problems is considered.

Given two positive constants $C_{1}$ and $C_{2}$, let us consider the
following radial problem
\begin{equation}
\!\!\!\left\{  \!%
\begin{array}
[c]{ll}%
\!-\operatorname{div}\left(  \!\dfrac{\Phi_{\blacklozenge}\left(  \left\vert
\nabla v\right\vert \right)  }{\left\vert \nabla v\right\vert ^{2}}\nabla
v\!\right)  \!=\!C_{1}\left(  \!f^{\bigstar}\left(  x\right)
-\operatorname{div}\left(  \!\Phi_{\blacklozenge\bullet}^{-1}\left(  G\left(
\omega_{N}\left\vert x\right\vert ^{N}\right)  \!\right)  \frac{x}{\left\vert
x\right\vert }\!\right)  \!\right)  \! & \!\mbox{ in }\Omega^{\bigstar}\!\\
& \\
\!v=0 & \!\mbox{ on }\partial\Omega^{\bigstar},\!
\end{array}
\!\right.  \! \label{Prob_sim}%
\end{equation}
where $\Omega^{\bigstar}$ is the ball centered in the origin having the same
measure as $\Omega,$ $f^{\bigstar}$ is the symmetric decreasing rearrangement
of $f$, $G$ is the pseudo-rearrangement of $\Phi_{\bullet}\left(
C_{2}\;g\right)  $ with respect to $u$, $\Phi_{\blacklozenge\bullet}^{-1}$
denotes the inverse of the Young conjugate of $\Phi_{\blacklozenge}$ and
$\omega_{N}$ is the measure of the $N$-dimensional unit ball$.$

It is easily to verified that the solution $v\left(  x\right)  $ to problem
(\ref{Prob_sim}) there exists if and only if either%
\begin{equation}
\underset{r\rightarrow+\infty}{\lim}\Psi_{\blacklozenge}\left(  r\right)
=+\infty\label{cond codominio1}%
\end{equation}
or%
\begin{equation}
C_{1}\left(  \dfrac{r^{1/N}f^{\ast\ast}(r)}{N\omega_{N}^{1/N}}+\Phi
_{\blacklozenge\bullet}^{-1}\left(  G\left(  r\right)  \right)  \right)
<\text{\ }\underset{r\rightarrow+\infty}{\lim}\Psi_{\blacklozenge}\left(
r\right)  \text{ \ \ \ for every }r>0 \label{cond codominio 2}%
\end{equation}
and%
\begin{equation}
\displaystyle\int_{0}^{\left\vert \Omega\right\vert }\Phi_{\blacklozenge
}\left(  \Psi_{\blacklozenge}^{-1}\left(  C_{1}\dfrac{r^{1/N}f^{\ast\ast}%
(r)}{N\omega_{N}^{1/N}}+C_{1}\Phi_{\blacklozenge\bullet}^{-1}\left(  G\left(
r\right)  \right)  \right)  \right)  dr<+\infty\label{grad-finito2}%
\end{equation}
holds, where $\Psi_{\blacklozenge}$ is defined in (\ref{Psi}). We observe that
(\ref{grad-finito2}) is equivalent to%
\[
\displaystyle\int_{\Omega^{\bigstar}}\Phi_{\blacklozenge}\left(  \left\vert
\nabla v\right\vert \right)  dx<+\infty.
\]
The main result of this paper is the following comparison result.

\begin{theorem}
\label{th confronto} Let $\Phi:\mathbb{R}^{N}\rightarrow\left[  0,+\infty
\right[  $ be a $N$-dimensional Young function fulfilling (\ref{lim}) and
(\ref{psi incrising}). Suppose that conditions (\ref{elliti}), (\ref{dati}),
(\ref{cond codominio1}) and (\ref{grad-finito2}) hold. If u is a weak solution
to problem (\ref{Prob}), then there exists two positive constants $C_{1}$ and
$C_{2}$ (independent of $u$) such that
\begin{equation}
\!\!u^{\bigstar}\left(  x\right)  \!\leq\!v\left(  x\right)  \!\text{
\ \ \ for a.e. \ }x\in\Omega^{\bigstar}, \label{confronto}%
\end{equation}
where $v\left(  x\right)  $ is the spherically symmetric solution to problem
(\ref{Prob_sim}) given by
\begin{equation}
v\left(  x\right)  =\int_{\omega_{N}\left\vert x\right\vert ^{N}}^{\left\vert
\Omega\right\vert }\frac{1}{N\omega_{N}^{1/N}r^{1/N^{\prime}}}\Psi
_{\blacklozenge}^{-1}\left(  C_{1}\frac{r^{1/N}f^{\ast\ast}(r)}{N\omega
_{N}^{1/N}}+C_{1}\Phi_{\blacklozenge\bullet}^{-1}\left(  G\left(  r\right)
\right)  \right)  dr \label{v}%
\end{equation}
for $x\in\Omega^{\bigstar}$ and $G$ is the pseudo-rearrangement of
$\Phi_{\bullet}\left(  C_{2}\;g\right)  $ with respect to $u$. Moreover%
\begin{equation}
\int_{\Omega}\Phi\left(  \nabla u\left(  x\right)  \right)  dx\leq\int
_{\Omega^{\bigstar}}\Phi_{\blacklozenge}\left(  \left\vert \nabla v\left(
x\right)  \right\vert \right)  dx. \label{conf-grad}%
\end{equation}

\end{theorem}

\noindent

\begin{proof}
We define the functions $u_{\kappa,t}:\Omega\rightarrow$\textbf{ }$\mathbb{R}$
as
\[
u_{\kappa,t}\left(  x\right)  =\left\{
\begin{array}
[c]{ll}%
0 & \mbox{ if }\left\vert u\left(  x\right)  \right\vert \leq t,\vspace
{0.2cm}\\
\left(  \left\vert u\left(  x\right)  \right\vert -t\right)  \text{sign}%
\left(  u\left(  x\right)  \right)  & \mbox{ if }t<\left\vert u\left(
x\right)  \right\vert \leq t+\kappa\\
& \\
\kappa\;\text{sign}\left(  u\left(  x\right)  \right)  & \mbox{ if
}t+\kappa<\left\vert u\left(  x\right)  \right\vert ,
\end{array}
\right.
\]
for any fixed $\ t$ and $\kappa>0.$ This function can be choose as test
function in (\ref{sol deb}). Reasoning as in \cite{cianchi anisotropo} we
have
\begin{equation}
\int_{t<\left\vert u\right\vert <t+\kappa}\Phi\left(  \nabla u\right)
dx\geq\int_{t<u^{\bigstar}<t+\kappa}\Phi_{\blacklozenge}\left(  \left\vert
\nabla u^{\bigstar}\right\vert \right)  dx \label{A}%
\end{equation}
and by (\ref{elliti})
\begin{align}
\frac{1}{\kappa}\int_{t<\left\vert u\right\vert \leq t+\kappa}\Phi\left(
\nabla u\right)  dx  &  \leq\frac{1}{\kappa}\int_{t<\left\vert u\right\vert
\leq t+\kappa}f(x)\left(  \left\vert u\left(  x\right)  \right\vert -t\right)
\text{sign}\left(  u\left(  x\right)  \right)  \text{ }dx\label{B}\\
&  +\int_{\left\vert u\right\vert >t+\kappa}f(x)\text{sign}\left(  u\left(
x\right)  \right)  \text{ }dx+\frac{1}{\kappa}\int_{t<\left\vert u\right\vert
\leq t+\kappa}g\cdot\nabla u\text{ }dx.\nonumber
\end{align}
Letting $\kappa\rightarrow0^{+}$, we obtain%
\begin{equation}
-\frac{d}{dt}\int_{\left\vert u\right\vert >t}\Phi\left(  \nabla u\right)
dx\leq\int_{0}^{\mu_{u}\left(  t\right)  }f^{\ast}\left(  s\right)
ds-\frac{d}{dt}\int_{\left\vert u\right\vert >t}g\cdot\nabla u\text{ }dx.
\label{C}%
\end{equation}
Using Young inequality (\ref{Young}) and inequality (\ref{costante}) in
(\ref{C}), we have%
\[
-\frac{d}{dt}\int_{\left\vert u\right\vert >t}\Phi\left(  \nabla u\right)
dx\leq\int_{0}^{\mu_{u}\left(  t\right)  }f^{\ast}\left(  s\right)
ds-\frac{d}{dt}\int_{\left\vert u\right\vert >t}\Phi_{\bullet}\left(
\frac{g\left(  x\right)  }{\varepsilon}\right)  dx-\varepsilon\frac{d}{dt}%
\int_{\left\vert u\right\vert >t}\Phi\left(  \nabla u\right)  dx
\]
for every $0<\varepsilon<1.$ Then we get%
\begin{equation}
-\left(  1-\varepsilon\right)  \frac{d}{dt}\int_{\left\vert u\right\vert
>t}\Phi\left(  \nabla u\right)  dx\leq\int_{0}^{\mu_{u}\left(  t\right)
}f^{\ast}\left(  s\right)  ds+G\left(  \mu_{u}\left(  t\right)  \right)
\left(  -\mu_{u}^{\prime}\left(  t\right)  \right)  , \label{eq_2}%
\end{equation}
where $G$ is the pseudo-rearrangement of $\Phi_{\bullet}\left(  \frac{g\left(
x\right)  }{\varepsilon}\right)  $ with respect to $u$.

Since $\Phi_{\blacklozenge}$ is a strictly monotone function in $\left[
s_{0},+\infty\right[  ,$ we claim that%
\begin{equation}
1\leq\frac{-\mu_{u}^{\prime}\left(  t\right)  }{N\omega_{N}^{1/N}\left(
\mu_{u}\left(  t\right)  \right)  ^{1/N^{\prime}}}\Phi_{\blacklozenge}%
^{-1}\left(  \frac{-\frac{d}{dt}\int_{\left\vert u\right\vert >t}\Phi\left(
\nabla u\right)  dx}{-\mu_{u}^{\prime}\left(  t\right)  }\right)
\label{Cianchi_Dis}%
\end{equation}
for $t>0$, where $s_{0}$ is defined in \eqref{s0}. Indeed, by Jensen
inequality and by $\mu_{u^{\bigstar}}=\mu_{u}$, it follows that
\[
\Phi_{\blacklozenge}\left(  \frac{\frac{1}{\kappa}\int_{\left\{
t<u^{\bigstar}<t+\kappa\right\}  }\left\vert \nabla u^{\bigstar}\right\vert
dx}{\frac{\mu_{u}\left(  t\right)  -\mu_{u}\left(  t+\kappa\right)  }{\kappa}%
}\right)  \leq\frac{\frac{1}{\kappa}\int_{\left\{  t<u^{\bigstar}%
<t+\kappa\right\}  }\Phi_{\blacklozenge}\left(  \left\vert \nabla u^{\bigstar
}\right\vert \right)  dx}{\frac{\mu_{u}\left(  t\right)  -\mu_{u}\left(
t+\kappa\right)  }{\kappa}}%
\]
\noindent for $t,\kappa>0.$ Let us denote by $\Phi_{\blacklozenge}^{-1}$ the
inverse of $\Phi_{\blacklozenge}$ restricted to $[s_{0},+\infty)$. By the
strictly monotonicity of $\Phi_{\blacklozenge}^{-1}$ on $[0,+\infty)$, Coarea
formula and recalling that the level set of $u^{\bigstar}$ are balls, we get%
\[
\frac{1}{\kappa}\int_{t}^{t+\kappa}N\omega_{N}^{1/N}\mu_{u}\left(  r\right)
^{\frac{1}{N^{\prime}}}dr\leq\frac{\mu_{u}\left(  t\right)  -\mu_{u}\left(
t+\kappa\right)  }{\kappa}\Phi_{\blacklozenge}^{-1}\left(  \frac{\frac
{1}{\kappa}\int_{\left\{  t<u^{\bigstar}<t+\kappa\right\}  }\Phi
_{\blacklozenge}\left(  \left\vert \nabla u^{\bigstar}\right\vert \right)
dx}{\frac{\mu_{u}\left(  t\right)  -\mu_{u}\left(  t+\kappa\right)  }{\kappa}%
}\right)  \text{ \ for }t,\kappa>0.
\]
Using again the monotonicity of $\Phi_{\blacklozenge}^{-1}$ and (\ref{A}), we
have%
\[
\frac{1}{\kappa}\int_{t}^{t+\kappa}N\omega_{N}^{1/N}\mu_{u}\left(  r\right)
^{\frac{1}{N^{\prime}}}dr\leq\frac{\mu_{u}\left(  t\right)  -\mu_{u}\left(
t+\kappa\right)  }{\kappa}\Phi_{\blacklozenge}^{-1}\left(  \frac{\frac
{1}{\kappa}\int_{\left\{  t<\left\vert u\right\vert <t+\kappa\right\}  }%
\Phi\left(  \nabla u\right)  dx}{\frac{\mu_{u}\left(  t\right)  -\mu
_{u}\left(  t+\kappa\right)  }{\kappa}}\right)  \text{ \ for }t,\kappa>0.
\]
Letting $\kappa\rightarrow0^{+},$ we obtain \eqref{Cianchi_Dis}. Using
\eqref{Cianchi_Dis} in (\ref{eq_2}), we have%
\begin{gather}
-(1-\varepsilon)\frac{d}{dt}\int_{\left\vert u\right\vert >t}\Phi\left(
\nabla u\right)  dx\leq\frac{-\mu_{u}^{\prime}\left(  t\right)  }{N\omega
_{N}^{1/N}\left(  \mu_{u}\left(  t\right)  \right)  ^{1/N^{\prime}}}%
\Phi_{\blacklozenge}^{-1}\left(  \frac{-\frac{d}{dt}\int_{\left\vert
u\right\vert >t}\Phi\left(  \nabla u\right)  dx}{-\mu_{u}^{\prime}\left(
t\right)  }\right)  \int_{0}^{\mu_{u}\left(  t\right)  }f^{\ast}\left(
s\right)  ds\label{dis}\\
+G\left(  \mu_{u}\left(  t\right)  \right)  \left(  -\mu_{u}^{\prime}\left(
t\right)  \right)  :=I_{1}+I_{2}.\nonumber
\end{gather}
Using inequality (\ref{costante}) with $\digamma=-\Phi_{\blacklozenge}^{-1}$
and Young inequality (\ref{Young}) we get%
\begin{equation}
I_{1}\leq-\mu_{u}^{\prime}\left(  t\right)  \Phi_{\blacklozenge\bullet}\left(
\frac{\int_{0}^{\mu_{u}\left(  t\right)  }f^{\ast}\left(  s\right)
ds}{\varepsilon N\omega_{N}^{1/N}\left(  \mu_{u}\left(  t\right)  \right)
^{1/N^{\prime}}}\right)  -\varepsilon\,\frac{d}{dt}\int_{\left\vert
u\right\vert >t}\Phi\left(  \nabla u\right)  dx. \label{I1}%
\end{equation}
Choosing $0<\varepsilon<\frac{1}{2}$ and using (\ref{I1}), inequality
(\ref{dis}) becomes
\[
-\left(  1-2\,\varepsilon\right)  \frac{d}{dt}\int_{\left\vert u\right\vert
>t}\Phi\left(  \nabla u\right)  dx\leq-\mu_{u}^{\prime}\left(  t\right)
\Phi_{\blacklozenge\bullet}\left(  \frac{\int_{0}^{\mu_{u}\left(  t\right)
}f^{\ast}\left(  s\right)  ds}{\varepsilon N\omega_{N}^{1/N}\left(  \mu
_{u}\left(  t\right)  \right)  ^{1/N^{\prime}}}\right)  +G\left(  \mu
_{u}\left(  t\right)  \right)  \left(  -\mu_{u}^{\prime}\left(  t\right)
\right)  .
\]
Now using (\ref{sublineare}) and (\ref{costante}) we get%
\begin{equation}
-\frac{d}{dt}\int_{\left\vert u\right\vert >t}\Phi\left(  \nabla u\right)
dx\leq-\mu_{u}^{\prime}\left(  t\right)  \Phi_{\blacklozenge\bullet}\left(
\frac{1}{\left(  1-2\,\varepsilon\right)  }\left(  \frac{\int_{0}^{\mu
_{u}\left(  t\right)  }f^{\ast}\left(  s\right)  ds}{\varepsilon N\omega
_{N}^{1/N}\left(  \mu_{u}\left(  t\right)  \right)  ^{1/N^{\prime}}}%
+\Phi_{\blacklozenge\bullet}^{-1}\left(  G\left(  \mu_{u}\left(  t\right)
\right)  \right)  \right)  \right)  , \label{phi}%
\end{equation}
where $\Phi_{\blacklozenge\bullet}^{-1}$\ is the inverse function of Young
conjugate of $\Phi_{\blacklozenge}$ . Then
\begin{equation}
-\frac{d}{dt}\int_{\left\vert u\right\vert >t}\Phi\left(  \nabla u\right)
dx\leq-\mu_{u}^{\prime}\left(  t\right)  \Phi_{\blacklozenge\bullet}\left(
C_{1}\frac{\int_{0}^{\mu_{u}\left(  t\right)  }f^{\ast}\left(  s\right)
ds}{N\omega_{N}^{1/N}\left(  \mu_{u}\left(  t\right)  \right)  ^{1/N^{\prime}%
}}+C_{1}\Phi_{\blacklozenge\bullet}^{-1}\left(  G\left(  \mu_{u}\left(
t\right)  \right)  \right)  \right)  \label{phi2}%
\end{equation}
for some positive constant $C_{1}$. By (\ref{Cianchi_Dis}) and (\ref{phi2}) we
have%
\[
1\leq\frac{-\mu_{u}^{\prime}\left(  t\right)  }{N\omega_{N}^{1/N}\left(
\mu_{u}\left(  t\right)  \right)  ^{1/N^{\prime}}}\Phi_{\blacklozenge}%
^{-1}\left(  \Phi_{\blacklozenge\bullet}\left(  C_{1}\frac{\int_{0}^{\mu
_{u}\left(  t\right)  }f^{\ast}\left(  s\right)  ds}{N\omega_{N}^{1/N}\left(
\mu_{u}\left(  t\right)  \right)  ^{1/N^{\prime}}}+C_{1}\Phi_{\blacklozenge
\bullet}^{-1}\left(  G\left(  \mu_{u}\left(  t\right)  \right)  \right)
\right)  \right)  .
\]
Now using (\ref{dis PsiRombo}), we get%
\begin{equation}
1\leq\frac{-\mu_{u}^{\prime}\left(  t\right)  }{N\omega_{N}^{1/N}\left(
\mu_{u}\left(  t\right)  \right)  ^{1/N^{\prime}}}\Psi_{\blacklozenge}%
^{-1}\left(  C_{1}\frac{\int_{0}^{\mu_{u}\left(  t\right)  }f^{\ast}\left(
s\right)  ds}{N\omega_{N}^{1/N}\left(  \mu_{u}\left(  t\right)  \right)
^{1/N^{\prime}}}+C_{1}\Phi_{\blacklozenge\bullet}^{-1}\left(  G\left(  \mu
_{u}\left(  t\right)  \right)  \right)  \right)  \text{ \ } \label{1<}%
\end{equation}
for a.e. $t>0.$ In a standard way we obtain (\ref{confronto}).

As regards (\ref{conf-grad}), if we integrate (\ref{phi2}), we get%
\[
\int_{\Omega}\Phi\left(  \nabla u\right)  dx\leq\int_{0}^{\left\vert
\Omega\right\vert }\Phi_{\blacklozenge\bullet}\left(  C_{1}\frac
{r^{1/N}f^{\ast\ast}(r)}{N\omega_{N}^{1/N}}+C_{1}\Phi_{\blacklozenge\bullet
}^{-1}\left(  G\left(  r\right)  \right)  \right)  dr.
\]
Then using (\ref{dis PsiRombo}) we have%
\[
\int_{\Omega}\Phi\left(  \nabla u\right)  dx\leq\int_{0}^{\left\vert
\Omega\right\vert }\Phi_{\blacklozenge}\left(  \Psi_{\blacklozenge}%
^{-1}\left(  C_{1}\frac{r^{1/N}f^{\ast\ast}(r)}{N\omega_{N}^{1/N}}+C_{1}%
\Phi_{\blacklozenge\bullet}^{-1}\left(  G\left(  r\right)  \right)  \right)
\right)  dr=\int_{\Omega^{\bigstar}}\Phi_{\blacklozenge}\left(  \left\vert
\nabla v\right\vert \right)  dx,
\]
namely (\ref{conf-grad}).
\end{proof}

\begin{remark}
\emph{In Theorem \ref{th confronto} condition (\ref{cond codominio1}) can be
replaced by (\ref{cond codominio 2}) with the extra hypothesis that }%
$\Phi_{\blacklozenge\bullet}$\emph{ is a Young function. This additional
assumption is due to the presence of the divergence term }$g$\emph{ that
forces us to use Young inequality.}
\end{remark}

\begin{remark}
\emph{In view of equivalence \eqref{equiv}, an analogous result as in Theorem
\ref{th confronto} can be obtained using the symmetric increasing
rearrangement $\Phi_{\bigstar}$ instead of the Klimov rearrangement
$\Phi_{\blacklozenge}$ of $\Phi$. }
\end{remark}

Theorem \ref{th confronto} can be used to prove \textit{a priori} bound for
solutions to problem \eqref{Prob}. We note that norm estimates for $u$ in
\textit{r.i.} space are less easily expressed in terms of \textit{r.i.} norm
of
\begin{equation}
F(r)=\Psi_{\blacklozenge}^{-1}\left(  C_{1}\frac{r^{1/N}f^{\ast\ast}%
(r)}{N\omega_{N}^{1/N}}+C_{1}\Phi_{\blacklozenge\bullet}^{-1}\left(  G\left(
r\right)  \right)  \right)  \text{ \ \ for \ }r>0, \label{F}%
\end{equation}
rather than in terms of norms of $f$ and $g.$ Indeed any \textit{r.i.} norm of
$F$ defines a \textit{r.i.} functional of $f$ and $g$, which is again a norm,
or a quasi norm in several important instances.

When $F$ belongs to the Orlicz space $L^{A}(\Omega)$ for suitable
one-dimensional Young function $A,$ (\ref{confronto}) and \cite[Inequality
2.7]{cianchi 97} yield $u$ belongs to another Orlicz space $L^{A_{n}}%
(\Omega).$ The function $A_{n}$ is the Young function associated with $A$ by
$A_{n}\left(  s\right)  =$ $A(H_{A}^{-1}(\left\vert s\right\vert ))$ with
$s\in\mathbb{R}$, where $H_{A}$ is defined as in (\ref{H}) with $\Phi
_{\blacklozenge}$ replaced by $A.$ \newline We will prove that $u$ enjoys a
stronger summability which is given by the finiteness of the norm in the
Orlicz-Lorentz $X_{A,N}(\Omega).$

\begin{proposition}
\label{prop_3.7} Under the same assumptions as in Theorem \ref{th confronto},
let $u$ be a weak solution to problem \eqref{Prob}. Assume that the function
$F$, defined by \eqref{F}, belongs to $L^{A}(0,|\Omega|)$ for some
one-dimensional Young function $A$ satisfying \eqref{int_A}.

\noindent(i) If $\displaystyle\int^{\infty}\left(  \frac{s}{A(s)}\right)
^{\frac{1}{N-1}}\,ds=+\infty$, then
\begin{equation}
\Vert u\Vert_{X_{A,N}(\Omega)}\leq c\Vert F\Vert_{L^{A}(0, |\Omega|)}
\label{stima_X_An}%
\end{equation}
for some positive constant $c$ independent of $u$, $f$ and $g_{i}$ for $i=1,
\ldots, N$, where $\Vert\cdot\Vert_{X_{A,N}(\Omega)}$ is defined in
(\ref{O-L_norm}).

\noindent(ii) If $\displaystyle\int^{\infty}\left(  \frac{s}{A(s)}\right)
^{\frac{1}{N-1}}\,ds<+\infty$, then
\begin{equation}
\Vert u\Vert_{L^{\infty}(\Omega)}\leq c\left(  \int_{0}^{\infty}%
\frac{A_{\bullet}(s)}{s^{N^{\prime}+1}}\;ds\right)  ^{\frac{1}{N^{\prime}}%
}\Vert F\Vert_{L^{A}(0, |\Omega|)} \label{stima_X_An}%
\end{equation}
for some positive constant $c$ independent of $u$, $f$ and $g_{i}$ for $i=1,
\ldots, N$.
\end{proposition}

Proposition \ref{prop_3.7} is the analogous version of Proposition 3.7 in
\cite{cianchi anisotropo}, but the function $F$, defined by \eqref{F},
contains not only a part related to $f$ as in \cite{cianchi anisotropo}, but
also one related to $g_{i}$ for $i=1,\ldots,N$.

\bigskip

\section{ Applications and Examples}

\subsection{Example 1}

Let us consider
\begin{equation}
\Phi\left(  \xi\right)  =\underset{i=1}{\overset{N}{%
{\displaystyle\sum}
}}\lambda_{i}\left\vert \xi_{i}\right\vert ^{p_{i}}\text{ \ \ \ for }\xi
\in\mathbb{R}^{N}, \label{fi pi}%
\end{equation}
with $\lambda_{i}>0$ and $p_{i}>1$ for any $i\in\left\{  1,\ldots,N\right\}
$. More precisely we consider the following class of problems
\begin{equation}
\left\{
\begin{array}
[c]{lll}%
Lu:=-\operatorname{div}(a(x,u,\nabla u))=f-\underset{i=1}{\overset{N}{%
{\displaystyle\sum}
}}\dfrac{\partial}{\partial x_{i}}g_{i}\left(  x\right)  &  & \text{in }%
\Omega\\
u=0 &  & \text{on }\partial\Omega,
\end{array}
\right.  \label{problema}%
\end{equation}
where $\Omega$ is a bounded open subset of $\mathbb{R}^{N}$ with $N\geq2,$
$a:\Omega\times\mathbb{R}\times\mathbb{R}^{N}\rightarrow\mathbb{R}^{N}$ is
Carath\'{e}odory function such that every component $a_{j}(x,s,\xi)$ of $a$
fulfills
\begin{equation}
a(x,s,\xi)\cdot\xi\geq\overset{N}{\underset{i=1}{\sum}}\lambda_{i}\left\vert
\xi_{i}\right\vert ^{p_{i}}\ \ , \label{ellitticita}%
\end{equation}%
\begin{equation}
\left\vert a_{j}(x,s,\xi)\right\vert \leq\beta\left[  h(x)+\left\vert
s\right\vert ^{\overline{p}}+\overset{N}{\underset{i=1}{\sum}}\left\vert
\xi_{i}\right\vert ^{p_{i}}\right]  ^{\frac{1}{p_{j}^{\prime}}}\text{ \ \ with
}\beta>0,h\in L^{1}(\Omega), \label{crescita a}%
\end{equation}%
\begin{equation}
\left[  a(x,s,\xi)-a(x,s,\xi^{\prime})\right]  \cdot\left(  \xi-\xi^{\prime
}\right)  >0\text{ \ for }\xi\neq\xi^{\prime} \label{monotonia}%
\end{equation}
and $\ \ $%
\begin{equation}
f\in L^{\left(  \overline{p}^{\ast}\right)  ^{\prime},\overline{p}^{\prime}%
}(\Omega)\text{ and }g_{i}\in L^{p_{i}^{\prime}}(\Omega)\text{ for }%
i=1,\ldots,N, \label{dati f e g}%
\end{equation}
where $\overline{p}$ is the harmonic mean of $p_{1},\ldots,p_{N}$ given by
(\ref{p barrato bis}) such that
\begin{equation}
\overline{p}<N. \label{p barrato}%
\end{equation}

In this section we are interested to study the existence and regularity of
solutions to problem (\ref{problema}), improving results contained in
\cite{castro} and putting the data in the Lorentz spaces. As regards the
uniqueness we refer to \textit{e.g.} \cite{antontsev-chipot-08},
\cite{DiNardo-Feo-Guibe}, \cite{DiNardo-Feo} and the bibliography therein.

The natural space into consider the solutions to problem (\ref{problema}) is
the anisotropic Sobolev space $W_{0}^{1,\overrightarrow{p}}(\Omega)$, that we
define as the closure of $C_{0}^{\infty}(\Omega)$ with respect to the norm
$\left\Vert u\right\Vert _{W_{0}^{1,\overrightarrow{p}}(\Omega)}=\overset
{N}{\underset{i=1}{\sum}}\left\Vert \partial_{x_{i}}u\right\Vert _{L^{p_{i}%
}(\Omega)}.$ Here, $\overrightarrow{p}$ stands for $(p_{1},\ldots,p_{N})$.

\begin{definition}
A weak solution (resp. distributional solution) to problem (\ref{problema}) is
a function $u\in W_{0}^{1,\overrightarrow{p}}(\Omega)$ such that (resp. $u\in
W_{0}^{1,1}(\Omega)$ such that $a(x,u,\nabla u)\in L^{1}\left(  \Omega\right)
$ and)
\begin{equation}
\int_{\Omega}a(x,u,\nabla u)\nabla\varphi\,dx=\int_{\Omega}\left(
f\varphi+\underset{i=1}{\overset{N}{%
{\displaystyle\sum}
}}g_{i}\left(  x\right)  \varphi_{x_{i}}\right)  dx\text{ \ \ \ \ }%
\forall\varphi\in\mathcal{D}(\Omega). \label{sol debole}%
\end{equation}

\end{definition}

By definition (\ref{fi rombo}) of $\Phi_{\blacklozenge},$ condition
(\ref{p barrato}) corresponds to require (\ref{existence_2}). If $\overline
{p}<N,$ then there is a continuous embedding of Sobolev space $W_{0}%
^{1,\overrightarrow{p}}(\Omega)$ into $L^{q}(\Omega)$ for $q\in\left[
1,\max(\overline{p}^{\ast},p_{N})\right]  $. Otherwise if $\overline{p}\geq N$
there is a continuous embedding $W_{0}^{1,\overrightarrow{p}}(\Omega)\subset
L^{q}(\Omega)$ for $q\in\left[  1,+\infty\right[  .$ \bigskip

When $\Phi\left(  \xi\right)  $ is given by (\ref{fi pi}), Theorem
\ref{th confronto} ensure the following result.

\begin{corollary}
\label{corollario}Let $\Omega$ be a bounded open subset of $\mathbb{R}^{N}$
with $N\geq2$ and suppose conditions (\ref{ellitticita})-(\ref{p barrato})
hold. If u is a weak solution to problem (\ref{problema}), there exists a
positive constant $C$ such that
\begin{equation}
u^{\bigstar}\left(  x\right)  \leq C v\left(  x\right)  \text{ \ \ \ for a.e.
\ }x\in\Omega^{\bigstar}, \label{confronto_anisotropo}%
\end{equation}
where $v\left(  x\right)  $ is the spherically symmetric solution to problem
(\ref{Prob simm pi}) given by
\begin{equation}
v\left(  x\right)  =\int_{\omega_{N|}x|^{N}}^{\left\vert \Omega\right\vert
}\frac{1}{\Lambda^{\frac{1}{\overline{p}-1}}N\omega_{N}^{1/N}t^{1/N^{\prime}}%
}\left(  \frac{t^{1/N}f^{\ast\ast}(t)}{N\omega_{N}^{1/N}}+\left(  \overline
{p}\Lambda\right)  ^{1/\overline{p}}\left(  \overline{p}^{\prime}G\left(
t\right)  \right)  ^{1/\overline{p}^{\prime}}\right)  ^{\frac{1}{\overline
{p}-1}}\!dt\qquad\hbox{\rm for}\;x\in\Omega^{\bigstar},
\label{soluzione_anisotropo}%
\end{equation}
where $G$ is the pseudo-rearrangement of $\overset{N}{\underset{i=1}{%
{\displaystyle\sum}
}}$ $\dfrac{(g_{i})^{p_{i}^{\prime}}} {p_{i}^{\prime}\;\;(\lambda_{i}%
p_{i})^{p_{i}^{\prime}/p_{i}}}$ with respect to $u$ and $\Lambda$ is defined
in (\ref{lapda}). \noindent Moreover we have%
\begin{equation}
\overset{N}{\underset{i=1}{%
{\displaystyle\sum}
}}\int_{\Omega}\left\vert \frac{\partial u}{\partial x_{i}}\right\vert
^{p_{i}}dx\leq C\,\int_{\Omega^{\bigstar}}\left\vert \nabla v\right\vert
^{\overline{p}}dx. \label{stima grad anisotropo}%
\end{equation}

\end{corollary}

\medskip

\noindent

\begin{remark}
\emph{The thesis of Corollary \ref{corollario} holds also if $p_{i}\geq1$ for
$i=1,\ldots,N$ and $\overline{p}>1.$ }
\end{remark}

We emphasize that the \textit{a priori} estimates (\ref{confronto_anisotropo})
allows us to obtain existence and regularity results only if the anisotropy is
concentrated, \textit{i.e.}%
\begin{equation}
\max_{i=1,\ldots,N} p_{i}<\overline{p}^{\ast}. \label{cond pN}%
\end{equation}
Indeed Corollary \ref{corollario} compares a solution to anisotropic problem
(\ref{problema}) with the solution to isotropic problem (\ref{Prob simm pi}),
but condition $\max_{i} p_{i}\geq\overline{p}^{\ast}$ does not occur if the
operator is of an isotropic type as in problem (\ref{Prob simm pi}).

The existence of a weak solution $u$ to problem (\ref{problema}) follows by
the classical theory on Leray-Lions operator (see \textit{e.g.} \cite{lions})
and conditions in (\ref{dati f e g}) on data are enough to have $\left\Vert
u\right\Vert _{W_{0}^{1,\overrightarrow{p}}(\Omega)}<+\infty.$ Indeed, by
inequality (\ref{stima grad anisotropo}), up to easy computation based on some
inequalities contained in Proposition \ref{Dis Hardy} above, we get the
following \textit{a priori} estimate
\[
\overset{N}{\underset{i=1}{%
{\displaystyle\sum}
}}\left\Vert \frac{\partial u}{\partial x_{i}}\right\Vert _{L^{p_{i}}(\Omega
)}^{p_{i}}\leq c\overset{N}{\underset{i=1}{%
{\displaystyle\sum}
}}\left(  \left\Vert f\right\Vert _{L^{\left(  \overline{p}^{\ast}\right)
^{\prime},\overline{p}^{\prime}}(\Omega)}^{\overline{p}^{\prime}}+\overset
{N}{\underset{j=1}{%
{\displaystyle\sum}
}}\left\Vert g_{j}\right\Vert _{L^{p_{j}^{\prime}}(\Omega)}^{p_{j}^{\prime}%
}\right)  ^{p_{i}},
\]
for some positive constant $c.$

Now, using (\ref{confronto_anisotropo}) we investigate how the summability of
a weak solution $u$ of problem (\ref{problema}) varies by improving the
summability of the data in the Lorentz-Zygmund spaces.

\begin{proposition}
\label{soluzioni deboli}Let us suppose (\ref{ellitticita})-(\ref{monotonia}%
),(\ref{p barrato}) and (\ref{cond pN}) hold.

i) If $f\in L^{m,\frac{\sigma}{\overline{p}-1}}(\Omega)$ and g$_{i}\in
L^{r_{i},s_{i}}(\Omega)$ for $i=1,\ldots,N$ with
\[
\left\{
\begin{array}
[c]{ll}%
m>N/{\overline{p}} & \text{ \ and \ }0<\sigma\leq+\infty\\
& or\\
m=N/{\overline{p}} & \text{ and \ \ }0<\sigma\leq1
\end{array}
\right.  \text{ \ \ and \ \ }\left\{
\begin{array}
[c]{ll}%
r_{i}>Np_{i}^{\prime}/\overline{p} & \text{ \ \ and \ \ }0<s_{i}\leq+\infty\\
& or\\
r_{i}=Np_{i}^{\prime}/\overline{p} & \text{ \ \ and \ \ \ }0<s_{i}\leq
p_{i}^{\prime}/\overline{p},
\end{array}
\right.
\]
then there exists at least one bounded weak solution $u$ to problem
(\ref{problema}), such that%
\[
\left\Vert u\right\Vert _{L^{\infty}(\Omega)}\leq c\left(  \left\Vert
f\right\Vert _{L^{m,\frac{\sigma}{\overline{p}-1}}(\Omega)}^{\frac
{1}{\overline{p}-1}}+\overset{N}{\underset{i=1}{%
{\displaystyle\sum}
}}\left\Vert g_{i}\right\Vert _{L^{\frac{Np_{i}^{\prime}}{\overline{p}}%
,\frac{p_{i}^{\prime}}{\overline{p}}}(\Omega)}^{\frac{p_{i}^{\prime}%
}{\overline{p}}}\right)  .
\]

ii) If $f\in L^{\frac{N}{\overline{p}},\frac{\sigma}{\overline{p}-1}}(\Omega)$
and $g_{i}\in L^{\frac{Np_{i}^{\prime}}{\overline{p}},\frac{\sigma
p_{i}^{\prime}}{\overline{p}}}(\Omega)$ for $i=1,\ldots,N$ with $1<\sigma
\leq+\infty,$ then there exists at least one weak solution $u$ to problem
(\ref{problema}) belonging to $L^{\infty,\sigma}(\log L)^{-1}(\Omega)$, such
that%
\[
\left\Vert u\right\Vert _{L^{\infty,\sigma}(\log L)^{-1}(\Omega)}\leq c\left(
\left\Vert f\right\Vert _{L^{\frac{N}{\overline{p}},\frac{\sigma}{\overline
{p}-1}}(\Omega)}^{\frac{1}{\overline{p}-1}}+\overset{N}{\underset{i=1}{%
{\displaystyle\sum}
}}\left\Vert g_{i}\right\Vert _{L^{\frac{Np_{i}^{\prime}}{\overline{p}}%
,\frac{\sigma p_{i}^{\prime}}{\overline{p}}}(\Omega)}^{\frac{p_{i}^{\prime
}\sigma}{\overline{p}}}\right)  .
\]

iii) If $f\in L^{m,\frac{\sigma}{\overline{p}-1}}(\Omega)$ and g$_{i}\in
L^{^{\frac{mN\left(  \overline{p}-1\right)  }{N-m}\frac{p_{i}^{\prime}%
}{\overline{p}},\frac{\sigma p_{i}^{\prime}}{\overline{p}}}}(\Omega)$ \ for
$i=1,\ldots,N$ with either $\left(  \overline{p}^{\ast}\right)  ^{\prime
}<m<N/\overline{p}$ and $0<\sigma\leq+\infty$ or $m=\left(  \overline{p}%
^{\ast}\right)  ^{\prime}$ and $\sigma=\overline{p},$ then there exists at
least one weak solution $u$ to problem (\ref{problema}) belonging to
$L^{\frac{mN\left(  \overline{p}-1\right)  }{N-m\overline{p}},\sigma}%
(\Omega),$ such that%
\[
\left\Vert u\right\Vert _{L^{\frac{mN\left(  \overline{p}-1\right)
}{N-m\overline{p}},\sigma}(\Omega)}\leq c\left(  \left\Vert f\right\Vert
_{L^{m,\frac{\sigma}{\overline{p}-1}}(\Omega)}^{\frac{1}{\overline{p}-1}%
}+\overset{N}{\underset{i=1}{%
{\displaystyle\sum}
}}\left\Vert g_{i}\right\Vert _{L^{^{\frac{mN\left(  \overline{p}-1\right)
}{N-m}\frac{p_{i}^{\prime}}{\overline{p}},\frac{\sigma p_{i}^{\prime}%
}{\overline{p}}}}(\Omega)}^{\frac{\sigma p_{i}^{\prime}}{\overline{p}}%
}\right)  .
\]

In all cases c is a positive constant independent of u, f and $g_{i}$ for
$i=1,\ldots, N.$
\end{proposition}

\begin{remark}
\emph{We observe that the summability conditions given in Proposition
\ref{soluzioni deboli} are weaker than the conditions required in Proposition
\ref{prop_3.7}. More precisely, by Proposition \ref{prop_3.7} with
}$A(t)=t^{q}$\emph{ for }$1\leq q\leq N$\emph{, it follows that }%
\begin{equation}
\left\Vert u\right\Vert _{L^{q^{\ast},q}(\Omega)}\leq c\left\Vert F\right\Vert
_{L^{q}(0,\left\vert \Omega\right\vert )}\leq c\left(  \left\Vert f\right\Vert
_{L^{\frac{Nq}{q+N(\overline{p}-1)},\frac{q}{\overline{p}-1}}(\Omega
)}+\overset{N}{\underset{i=1}{{\displaystyle\sum}}}\left\Vert g_{i}\right\Vert
_{L^{\frac{p_{i}^{\prime}q}{\overline{p}}}(\Omega)}^{\frac{p_{i}^{\prime}%
q}{\overline{p}}}\right)  \label{stima}%
\end{equation}
\emph{for some positive constant }$c$\emph{. This result is not sharp in the
class of Lorentz-Zygmund spaces. Indeed, if for example }$q=\overline{p}%
$\emph{, then (\ref{stima}) gives that }$u\in L^{\overline{p}^{\ast}%
,\overline{p}}(\Omega)$\emph{ when (\ref{dati f e g}) is in force, whereas
Proposition \ref{soluzioni deboli} assures that }$u\in L^{\overline{p}^{\ast}%
}(\Omega)\subset L^{\overline{p}^{\ast},\overline{p}}(\Omega).$
\end{remark}

\bigskip

\begin{proof}
[Proof of Proposition \ref{soluzioni deboli}]By Corollary \ref{corollario} and
by Hardy--Littewood inequality we have%
\begin{align}
u^{\ast}\left(  s\right)   &  \leq c\left[
{\displaystyle\int_{s}^{\left\vert \Omega\right\vert }}
\frac{1}{N\omega_{N}^{1/N}t^{1/N^{\prime}}}\left(  \frac{t^{1/N}f^{\ast\ast
}(t)}{N\omega_{N}^{1/N}}\right)  ^{\overline{p}^{\prime}/\overline{p}}dt+%
{\displaystyle\int_{s}^{\left\vert \Omega\right\vert }}
\frac{\left(  G^{\ast}(t)\right)  ^{1/\overline{p}}}{N\omega_{N}%
^{1/N}t^{1/N^{\prime}}}dt\right] \nonumber\label{v1v2}\\
&  =:c\left[  v_{1}\left(  s\right)  +v_{2}\left(  s\right)  \right]
\end{align}
for some positive constant $c$, which can be vary from line to line.
Inequality \eqref{v1v2} implies
\[
\left\Vert u\right\Vert _{X}\leq c\left[  \left\Vert v_{1}\right\Vert
_{X}+\left\Vert v_{2}\right\Vert _{X}\right]  ,
\]
where $X$ is an appropriate Lorentz-Zygmund space which varies with respect to
the summability of data. As regard the first norm of the right-hand side
$\left\Vert v_{1}\right\Vert _{X},$ the thesis follows by Proposition 3.8 of
\cite{cianchi anisotropo}. Then in what follows we only take into account
$v_{2}$. For convenience of the reader here we exhibit only the proof in the
case \textit{(iii) }for $\left(  \overline{p}^{\ast}\right)  ^{\prime
}<m<N/\overline{p}$ and $0<\sigma<+\infty$.

By (\ref{Hardy2}) \ for increasing function and by (\ref{Hardy1}) for
decreasing function, it follows that
\begin{align*}
\left\Vert v_{2}\right\Vert _{L^{\frac{mN\left(  \overline{p}-1\right)
}{N-m\overline{p}},\sigma}}^{\sigma}  &  \leq c\int_{0}^{\left\vert
\Omega\right\vert }\left[  s^{\frac{N-m\overline{p}}{mN\left(  \overline
{p}-1\right)  }}\int_{s}^{\left\vert \Omega\right\vert }t^{-\frac{1}%
{N^{\prime}}}\left(  G^{^{\ast}}(t)\right)  ^{\frac{1}{\overline{p}}%
}dt\right]  ^{\sigma}\frac{ds}{s}\text{ \ }\\
&  \leq c\int_{0}^{\left\vert \Omega\right\vert }\left[  s^{\frac
{N-m\overline{p}}{mN\left(  \overline{p}-1\right)  }}\int_{s}^{\left\vert
\Omega\right\vert }t^{-\frac{1}{N^{\prime}}-\frac{1}{\overline{p}}}\left(
\int_{0}^{t}G^{^{\ast}}(\tau)d\tau\right)  ^{\frac{1}{\overline{p}}}dt\right]
^{\sigma}\frac{ds}{s}\\
&  \leq c\int_{0}^{\left\vert \Omega\right\vert }\left[  s^{\frac
{N-m\overline{p}}{mN\left(  \overline{p}-1\right)  }+\frac{1}{N}-\frac
{1}{\overline{p}}}\left(  \int_{0}^{t}G^{^{\ast}}(\tau)d\tau\right)
^{\frac{1}{\overline{p}}}\right]  ^{\sigma}\frac{ds}{s}\text{ }\\
&  \leq c\overset{N}{\underset{i=1}{%
{\displaystyle\sum}
}}\int_{0}^{\left\vert \Omega\right\vert }\left[  s^{\left(  \frac
{N-m\overline{p}}{mN\left(  \overline{p}-1\right)  }+\frac{1}{N}-\frac
{1}{\overline{p}}\right)  \overline{p}}\int_{0}^{s}\left(  g_{i}^{\ast}%
(\tau)\right)  ^{^{p_{i}^{\prime}}}d\tau\right]  ^{\frac{\sigma}{\overline{p}%
}}\frac{ds}{s}\\
&  \leq c\overset{N}{\underset{i=1}{%
{\displaystyle\sum}
}}\left\Vert g_{i}\right\Vert _{L^{^{\frac{mN\left(  \overline{p}-1\right)
}{N-m}\frac{p_{i}^{\prime}}{\overline{p}},\frac{\sigma p_{i}^{\prime}%
}{\overline{p}}}}(\Omega)}^{\frac{\overline{p}}{\sigma p_{i}^{\prime}}}.
\end{align*}
The other cases can be obtained using Theorem \ref{th confronto} and standard
weighted Hardy-type inequalities when $\sigma\geq1$ (see \textit{e.g.}
\cite{KP}), and some appropriate Hardy-type inequalities for monotone
functions when $0<\sigma<1$ (see \textit{e.g.} \cite{Goldmann2000}).
\end{proof}

\begin{remark}
\bigskip\emph{When }$g_{i}\equiv0$\emph{ for }$i=1,..,N$\emph{ Proposition
\ref{soluzioni deboli} gives the same results as \cite{castro} in the class of
Lebesgue spaces and the same regularity results of \cite{cianchi anisotropo}
in the class of Lorentz-Zygmund spaces.}
\end{remark}

Finally we are interested on the existence of distributional solutions to
problem (\ref{problema}) when the data $f\not \in L^{\left(  \overline
{p}^{\ast}\right)  ^{\prime},\overline{p}^{\prime}}(\Omega)$. \ For simplicity
we consider $g_{i}=0$ for $i=1,\ldots, N.$

\begin{proposition}
\label{soluzioni distribuzionali}Suppose (\ref{ellitticita})-(\ref{monotonia}%
), (\ref{p barrato}), (\ref{cond pN}) hold and $\frac{p_{i}}{\overline
{p}^{\prime}}>\frac{N}{N-1}$ for $i=1,\ldots,N.$

i) Let $f\in L^{m,r}(\Omega).$ If $1<m<\left(  \overline{p}^{\ast}\right)
^{\prime}$ and $0<r\leq m^{\ast}$ or $m=\left(  \overline{p}^{\ast}\right)
^{\prime}$ and $\overline{p}^{\prime}<r\leq m^{\ast},$ then there exists at
least one distributional solution $u$ to problem (\ref{problema}) such that
$\frac{\partial u}{\partial x_{i}}\in L^{q_{i}}(\Omega)$ with $q_{i}%
=\frac{p_{i}m^{\ast}}{\overline{p}^{\prime}}$ for $i=1,\ldots, N.$

ii) Let $f\in L^{1,r}\left(  \log L\right)  ^{\alpha}(\Omega).$ If either
$0<r\leq1^{\ast}$ and $\alpha\geq1$ or $r>1^{\ast}$ and $\frac{1}{r}%
+\alpha>\frac{1}{1^{\ast}}+1$ then there exists at least one distributional
solution $u$ to problem (\ref{problema}) such that $\frac{\partial u}{\partial
x_{i}}\in L^{q_{i}}(\Omega)$ with $q_{i}=\frac{p_{i}1^{\ast}}{\overline
{p}^{\prime}}$ \ for $\ i=1,\ldots,N.$
\end{proposition}

\begin{proof}
Using standard approximation method (see \emph{e.g.} \cite{Bo}), we proceed
considering a sequence of approximate problems
\begin{equation}
\left\{
\begin{array}
[c]{lll}%
Lu_{h}=f_{h} &  & \text{in }\Omega\\
u_{h}=0 &  & \text{on }\partial\Omega,
\end{array}
\right.  \label{prob app}%
\end{equation}
where $f_{h}$ are smooth enough in order to assure the existence of a weak
solution $u_{h}\in W_{0}^{1,\overrightarrow{p}}(\Omega)$ and $f_{h}\rightarrow
f$ in $L^{m,r}(\Omega)$ if $m>1$ or $f_{h}\rightarrow f$ in $L^{1,r}(\log
L)^{\alpha}(\Omega)$ if $m=1$ for $r$ and $\alpha$ as in the statement.

First we prove an estimate of the norm of $\frac{\partial u_{h}}{\partial
x_{i}}$ in terms of norm of data adapting the standard symmetrization method.
Arguing as in the proof of Theorem \ref{th confronto}\ we obtain the analogue
of (\ref{phi}), \textit{i.e.}%
\begin{equation}
-\frac{d}{dt}\int_{\left\vert u_{h}\right\vert >t}\overset{N}{\underset
{j=1}{\sum}}\left\vert \frac{\partial u_{h}}{\partial x_{j}}\right\vert
^{p_{j}}dx\leq c\left(  -\mu_{u_{h}}^{\prime}\left(  t\right)  \right)
\left(  \frac{\int_{0}^{\mu_{u_{h}}\left(  t\right)  }f_{h}^{\ast}\left(
s\right)  ds}{\left(  \mu_{u_{h}}\left(  t\right)  \right)  ^{1/N^{\prime}}%
}\right)  ^{\overline{p}^{\prime}}. \label{aa}%
\end{equation}
Here and in what follows $c$ is a constant that can be vary from line to line.

Fixed $i\in\left\{  1,\ldots,N\right\}  $ for $1\leq q_{i}<p_{i}$ by
H\"{o}lder inequality and by inequality $\left(  \ref{aa}\right)  ,$ it
follows%
\begin{align}
-\frac{d}{dt}\int_{\left\vert u_{h}\right\vert >t}\left\vert \frac{\partial
u_{h}}{\partial x_{i}}\right\vert ^{q_{i}}dx  &  \leq\left(  -\frac{d}{dt}%
\int_{\left\vert u_{h}\right\vert >t}\left\vert \frac{\partial u_{h}}{\partial
x_{i}}\right\vert ^{p_{i}}dx\right)  ^{\frac{q_{i}}{p_{i}}}\left(  -\mu
_{u_{h}}^{\prime}\left(  t\right)  \right)  ^{1-\frac{q_{i}}{p_{i}}}\text{
\ }\nonumber\\
&  \leq c\left(  \frac{\int_{0}^{\mu_{u_{h}}\left(  t\right)  }f_{h}^{\ast
}\left(  s\right)  ds}{\left(  \mu_{u_{h}}\left(  t\right)  \right)
^{1/N^{\prime}}}\right)  ^{^{\overline{p}^{\prime}}\frac{q_{i}}{p_{i}}}\left(
-\mu_{u_{h}}^{\prime}\left(  t\right)  \right)  . \label{22}%
\end{align}
Integrating (\ref{22}), we get
\begin{equation}
\left\Vert \frac{\partial u_{h}}{\partial x_{i}}\right\Vert _{L^{q_{i}}%
}^{q_{i}}\leq\left\{
\begin{array}
[c]{lll}%
c%
{\displaystyle\int_{0}^{\left\vert \Omega\right\vert }}
\left(  s^{-\frac{1}{N^{\prime}}+\frac{1}{m^{\ast}}}%
{\displaystyle\int_{0}^{s}}
f_{h}^{\ast}\left(  z\right)  dz\right)  ^{m^{\ast}}\dfrac{ds}{s}\leq
c\left\Vert f_{h}\right\Vert _{L^{m,m^{\ast}}(\Omega)}^{m^{\ast}} & \text{for}
& m\neq1\\
&  & \\
c%
{\displaystyle\int_{0}^{\left\vert \Omega\right\vert }}
\left(
{\displaystyle\int_{0}^{s}}
f_{h}^{\ast}\left(  z\right)  dz\right)  ^{1^{\ast}}\dfrac{ds}{s}\leq
c\left\Vert f_{h}\right\Vert _{L^{1,1^{\ast}}\left(  \log L\right)
^{1}(\Omega)}^{1^{\ast}} & \text{for} & m=1.
\end{array}
\right.  \text{\ \ } \label{estimate}%
\end{equation}
Inequality (\ref{estimate}) follows using (\ref{Hardy1}), some suitable Hardy
inequalities and putting $q_{i}=\frac{p_{i}m^{\ast}}{\overline{p}^{\prime}}$.
This means that $\left\Vert \frac{\partial u_{h}}{\partial x_{i}}\right\Vert
_{L^{q_{i}}}^{q_{i}}$ is uniformly bounded in $L^{q_{i}}\left(  \Omega\right)
$ with $q_{i}>1$ (that implies $\frac{p_{i}}{\overline{p}^{\prime}}>\frac
{N}{N-1}$). Then there exist a function $u$ and some subsequence, which we
still denote by $u_{h}$, such that $u_{h}\rightarrow u$ strongly in
$L^{\overline{p}}(\Omega)$. Moreover we have (see \cite{castro} for more
details)
\begin{equation}
\partial_{i}u_{h}\rightarrow\partial_{i}u\qquad\hbox{strongly in}\;\;L^{r_{i}%
}(\Omega)\;\text{for }r_{i}<q_{i}. \label{castro}%
\end{equation}
Now using standard argument it is possible to pass to the limit in
(\ref{prob app}).
\end{proof}

\begin{remark}
\emph{The proof of Proposition \ref{soluzioni distribuzionali} guarantees the
existence of a Solution Obtaining as Limit of Approximations (see
\textit{e.g.} \cite{DA} for definition) that is also distributional solution.
Indeed, this type of solution is the limit of a sequence of bounded weak
solutions to the approximated problems (\ref{prob app}), whose data are
regular enough and approach the data of problem (\ref{problema}) in some
sense. \newline As well known it is possible to consider other equivalent
notions of solutions as the entropy solution and the renormalized solution
introduced in \cite{BBGGPV} and \cite{LM} respectively (see also the
bibliography starting with them).}
\end{remark}

\subsection{Example 2}

Let us consider%
\[
\Phi\left(  \xi\right)  =\underset{i=1}{\overset{N}{%
{\displaystyle\sum}
}}\left\vert \xi_{i}\right\vert ^{p_{i}}(\log\left(  c+\left\vert \xi
_{i}\right\vert \right)  )^{\alpha_{i}}\text{ \ }%
\]
for $\xi\in\mathbb{R}^{N},$ where either $p_{i}>1$ and $\alpha_{i}\in%
\mathbb{R}
$ or $p_{i}=1$ and $\alpha_{i}\geq0$ and constant $c$ is large enough for
$\left\vert s\right\vert ^{p_{i}}\log^{\alpha_{i}}\left(  c+\left\vert
s\right\vert \right)  $ to be convex for $i=1,\ldots,N$. In order to assure
that $\Phi_{\blacklozenge\bullet}$ is a Young function we have to leave out
the case $p_{i}=1$ and $\alpha_{i}=0$ for $i=1,\ldots,N$.

\noindent If we choose
\[
A(s)\approx|s|^{\sigma}\left(  \log(c+|s|)\right)  ^{\gamma}%
\]
near infinity, where either $\sigma>1$ and $\gamma\in\mathbb{R}$, or
$\sigma=1$ and $\gamma>0$, conclusions of Proposition \ref{prop_3.7} hold
with
\[
B(s)\approx\left\{
\begin{array}
[c]{lll}%
|s|^{\sigma}\left(  \log(c+|s|)\right)  ^{\gamma} & \text{if} & \sigma<N\\
|s|^{N}\left(  \log(c+|s|)\right)  ^{\gamma-N} & \text{if} & \sigma=N\text{
and }\gamma<N-1\\
|s|^{N}\left(  \log(c+|s|)\right)  ^{-1}\log(c+\log(c+|s|))^{-N} & \text{if} &
\sigma=N\text{ and }\gamma=N-1
\end{array}
\right.
\]

\noindent near infinity, and hence
\[
X_{A,N}(\Omega)=\left\{
\begin{array}
[c]{lll}%
L^{\frac{\sigma N}{N-\sigma},\sigma}\left(  \log\,L\right)  ^{\frac{\gamma
}{\sigma}}(\Omega) & \text{if} & \sigma<N\\
L^{\infty,N}\left(  \log L\right)  ^{\frac{\gamma}{N}-1}(\Omega) & \text{if} &
\sigma=N\text{ and }\gamma<N-1\\
L^{\infty,N}\left(  \log L\right)  ^{-\frac{1}{N}}(\log\,\log\,L)^{-1}%
(\Omega) & \text{if} & \sigma=N\text{ and }\gamma=N-1
\end{array}
\right.
\]
up to equivalent norms.

\noindent Here, $L^{\infty,N}\left(  \log\,L\right)  ^{-\frac{1}{N}}%
(\log\,\log\,L)^{-1}(\Omega)$ denotes the generalized Lorentz-Zygmund space
equipped with the norm
\[
\Vert h\Vert_{L^{\infty,N}\left(  \log\,L\right)  ^{-\frac{1}{N}}(\log
\log\,L)^{-1}(\Omega)}=\left\Vert s^{-\frac{1}{N}}\left(  1+\log\left(
|\Omega|/s\right)  \right)  ^{-\frac{1}{N}}\left(  1+\log\left(  1+\log\left(
|\Omega|/s\right)  \right)  \right)  ^{-1}h^{\ast}(s)\right\Vert
_{L^{N}(0,|\Omega|)}.
\]

\section{Appendix}

\bigskip We recall some Hardy inequalities with fixed weights. For more
details and other Hardy-types inequalities we refer to \cite{Goldmann2000} and
\cite{KP} (see \cite{D-F} too).

\begin{proposition}
\label{Dis Hardy}Let $\psi$ be a nonnegative measurable function on
$(0,+\infty).$ Suppose $0<r,q<+\infty.$ \vspace{-1.0cm} \begin{item}
\item[(i)]
If  $\psi$ belongs to the cone of monotone  function, then
\begin{equation}
\left(  \!\int_{0}^{+\infty}\!\left(  t^{-r}\!\int_{0}^{t}
\psi(s)\text{ }ds\!\right)  ^{q}\text{ }\frac{dt}{t}\!\right)  ^{\frac{1}{q}%
}\!\!\leq\!c\text{ }\!\left(  \!\int_{0}^{+\infty}\!\left(  t^{1-r}%
\psi(t)\right)  ^{q}\frac{dt}{t}\!\right)  ^{\frac{1}{q}} \label{Hardy1}%
\end{equation}
and
\begin{equation}
\left(  \!\int_{0}^{+\infty}\!\left(  t^{r}\!\int_{t}^{+\infty}\!
\psi(s)\text{ }ds\!\right)  ^{q}\!\text{ }\frac{dt}{t}\!\right)  ^{\frac{1}%
{q}}\!\!\leq\!c\text{ }\!\left(  \!\int_{0}^{+\infty}\!\left(
t^{1+r}\psi(t)\right)  ^{q}\frac{dt}{t}\!\right)  ^{\frac{1}{q}} \label{Hardy2}%
\end{equation}
hold .
\item[(ii)]
If $1\leq q<+\infty$, then inequalities (\ref{Hardy1}) and
(\ref{Hardy2}) hold without any assumptions on $\psi$.
\end{item}

\noindent In all cases, the constants $c$ are independent of $\psi$.
\end{proposition}

\paragraph*{Acknowledgements}

This work has been partially supported by GNAMPA of INdAM.


\begin{thebibliography}{999999}                                                                                           %


\bibitem[AF]{AF}E. Acerbi, N. Fusco, \emph{Partial regularity under
anisotropic }$(p,q)$\emph{ conditions}, J. Diff. Equat. \textbf{107} (1994), 46-67.

\bibitem[A]{A}A. Alberico, \emph{Boundedness of solutions to anisotropic
variational problems}, Comm. Part. Diff. Eq. \textbf{36} (2011), 470--486.

\bibitem[AC]{AC}A. Alberico, A. Cianchi, \emph{Comparison estimates in
anisotropic variational problems}, Manuscripta Math. \textbf{126} (2008), 481--503.

\bibitem[AFTL]{AFTL}A. Alvino, V. Ferone, G. Trombetti, P. L. Lions,
\emph{Convex symmetrization and applications}, Ann. Inst. H. Poincar\'{e}
Anal. Non Lin\'{e}aire \textbf{14} (1997), 275--293.

\bibitem[AT]{AT}A. Alvino, G. Trombetti, \emph{The best majorization constants
for a class of degenerate elliptic equations}, (Italian) Ricerche Mat.
\textbf{27} (1978), 413--428.

\bibitem[ACh]{antontsev-chipot-08}S. Antontsev, M. Chipot, \emph{Anisotropic
equations: uniqueness and existence results}, Diff. Int. Eq. \textbf{21}
(2008), 401--419.

\bibitem[ADS]{ADS}S. N. Antontsev, J. I. D\'{\i}az, S. Shmarev,\textit{
\emph{Energy methods for free boundary problems. Applications to nonlinear
PDEs and fluid mechanics}.} Progress in Nonlinear Differential Equations and
their Applications, 48. Birkh\"{a}user Boston, Inc., Boston, 2002.

\bibitem[BBGGPV]{BBGGPV}P. B\'{e}nilan, L. Boccardo, T. Gallou\"{e}t, R.
Gariepy, M. Pierre, J. L. V\'{a}zquez, \emph{An L1-theory of existence and
uniqueness of solutions of nonlinear elliptic equations}, Ann. Scuola Norm.
Sup. Pisa Cl. Sci. \textbf{22} (1995), 241--273.

\bibitem[BFK]{BFK}M. Belloni, V. Ferone, B. Kawohl, \emph{Isoperimetric
inequalities, Wulff shape and related questions for strongly nonlinear
elliptic equations,} Zeit. Angew. Math. Phys. (ZAMP) \textbf{ 54} (2003), 771-789.

\bibitem[BK]{BK}M. Bendahmane, K. Karlsen, \emph{Anisotropic doubly nonlinear
degenerate parabolic equations}\textit{.} Numerical mathematics and advanced
applications, 381--386, Springer, Berlin, 2006.

\bibitem[BS]{BS}C. Bennett, R. Sharpley, \emph{Interpolation of operators},
Pure and Applied Mathematics, 129, Academic Press, Inc., Boston, MA, 1988.

\bibitem[BMS]{BMS}L. Boccardo, P. Marcellini, C. Sbordone,\emph{ }$L^{\infty}%
$\emph{-regularity for variational problems with sharp nonstandard growth
conditions}, Boll. Un. Mat. Ital. A \textbf{4} (1990), 219-225.

\bibitem[C1]{cianchi 97}A. Cianchi, \emph{Boundedness of solutions to
variational problems under general growth conditions}\textit{,} Comm. Part.
Diff. Eq. \textbf{22} (1997), 1629--1646.

\bibitem[C2]{Clocal}A.Cianchi, \emph{Local boundedness of minimizers of
anisotropic functionals}, Ann. Inst. Henri Poincar\'{e}, Analyse non
lin\'{e}aire \textbf{17} (2000), 147-168.

\bibitem[C3]{cianchi immersione}A. Cianchi,\emph{ A fully anisotropic Sobolev
inequality}\textit{,} Pacific J. Math. \textbf{196} (2000), 283--295.

\bibitem[C4]{cianchi anisotropo}A. Cianchi, \emph{Symmetrization in
anisotropic elliptic problems}\textit{,} Comm. Part. Diff. Eq. \textbf{32}
(2007), 693--717.

\bibitem[Da]{DA}A. Dall'Aglio, \emph{Approximated solutions of equations with
L}$^{1}$\emph{ data. Application to the H-convergence of quasi-linear
parabolic equations}, Ann. Mat. Pura Appl. \textbf{170} (1996), 207--240.



\bibitem[DB]{DPdB}F. Della Pietra, G. di Blasio, \emph{Blow-up solutions for
some nonlinear elliptic equations involving a Finsler-Laplacian}, arXiv:1502.06768.

\bibitem[DG]{DPG}F. Della Pietra, N. Gavitone, \emph{Anisotropic elliptic
equations with general growth in the gradient and Hardy-type potentials}, J.
Differential Equations \textbf{255} (2013), 3788--3810.

\bibitem[dC]{castro}A. Di Castro, \emph{Existence and regularity results for
anisotropic elliptic problems}\textit{,} Adv. Nonlinear Stud. \textbf{9}
(2009), 367--393.

\bibitem[BF]{D-F}G. di Blasio, F. Feo,\emph{ A class of nonlinear degenerate
elliptic equations related to the Gauss measure}\textit{,} J. Math. Anal.
Appl. \textbf{386} (2012), 763--779.

\bibitem[BG]{Bo}L. Boccardo, T. Gallou\"{e}t, \emph{Nonlinear elliptic
equations with right-hand side measures}, Comm. Partial Differential Equations
\textbf{17} (1992), 641--655.

\bibitem[DFG]{DiNardo-Feo-Guibe}R. Di Nardo, F. Feo, O. Guib\'{e},
\emph{Uniqueness result for nonlinear anisotropic elliptic equations}%
\textit{.} Adv. Differential Equations \textbf{18} (2013), 433--458.

\bibitem[DF]{DiNardo-Feo}R. Di Nardo, F. Feo, \emph{Existence and uniqueness
for nonlinear anisotropic elliptic equations}\textit{, }Arch. Math. (Basel)
\textbf{102} (2014), 141--153.

\bibitem[ELM]{ELM}L. Esposito, F. Leonetti, G. Mingione, \emph{Sharp
regularity for functionals with }$(p,q)$\emph{ growth}, J.Diff. Equat.
\textbf{204} (2004), 5-55.

\bibitem[FGK]{FGK}I. Fragal\`{a}, F. Gazzola, B. Kawohl, \emph{Existence and
nonexistence results for anisotropic quasilinear elliptic equations}, Ann.
Inst. Henri Poincar\'{e}, Analyse non lin\'{e}aire \textbf{21} (2004), 715-734.

\bibitem[FGL]{FGL}I. Fragal\`{a}, F. Gazzola, G. Liebermann, \emph{Regularity
and nonexistence results for anisotropic quasilinear elliptic equations in
convex domains}\textit{,} Disc. Cont. Dynam. Syst. (2005), 280-286.

\bibitem[FS]{FS}N. Fusco, C. Sbordone, \emph{Some remarks on the regularity of
minima of anisotropic integrals}, Comm. Part. Diff. Equat. \textbf{18} (1993), 153-167.

\bibitem[Gi]{Gi}M. Giaquinta, \emph{Growth conditions and regularity, a
counterexample}, Manus. Math. \textbf{59} (1987), 245-248.

\bibitem[Go]{Goldmann2000}M. L. Goldman, \emph{Sharp estimates for the norms
of Hardy-type operators on cones of quasimonotone functions}\textit{.}
(Russian) Tr. Mat. Inst. Steklova \textbf{232} (2001), Funkts. Prostran.,
Garmon. Anal., Differ. Uravn., 115--143; translation in Proc. Steklov Inst.
Math. 2001, n. 1 (232), 109--137

\bibitem[H]{Ho}M. C. Hong, \emph{Some remarks on minimizers of variational
integrals with non standard growth conditions}, Boll. Un. Mat. Ital.
\textbf{6-A} (1992), 91-101.

\bibitem[Le]{Le}F. Leonetti, \emph{Weak differentiability for solutions to
nonlinear elliptic systems with }$(p,q)$\emph{ growth conditions}\textit{,}
Ann. Mat. Pura. Appl. \textbf{162} (1992), 349-366.



\bibitem[Ls]{lions}J. L. Lions, \emph{Quelques m\'{e}thodes de r\'{e}solution
des probl\`{e}mes aux limites non lin\'{e}aires}\textit{,} Dunod et
Gauthier-Villars, Paris, 1969.

\bibitem[LM]{LM}P.-L. Lions and F. Murat,\emph{ Sur les solutions
renormalis\'{e}es d'\'{e}quations elliptiques non lin\'{e}aires}, manuscript.

\bibitem[M]{Mar}P. Marcellini, \emph{Regularity of minimizers of integrals of
the calculus of variations with non standard growth conditions}, Arch. Rat.
Mech. Anal. \textbf{105} (1989), 267-284.

\bibitem [MRST]{mercaldo}A. Mercaldo, J.D. Rossi, S. Segura de Le\'{o}n, C.
Trombetti, \emph{Anisotropic p,q-Laplacian equations when p goes to 1},
Nonlinear Anal. \textbf{73} (2010), 3546--3560.

\bibitem[KP]{KP}A. Kufner and L. Persson, \emph{Weighted inequalities of Hardy
type}, World Scientific Publishing Co., Inc., River Edge, NJ, 2003.

\bibitem[K]{Klimov 74}V. S. Klimov, \emph{Isoperimetric inequalities and
imbedding theorems}, (Russian) Dokl. Akad. Nuak SSSR \textbf{217} (1974), 272--275.

\bibitem[RR]{rao}M. M. Rao, Z. D. Ren, \emph{Applications of Orlicz spaces},
Monographs and Textbooks in Pure and Applied Mathematics, 250, Marcel Dekker,
Inc., New York, 2002.



\bibitem[S]{St2}B. Stroffolini, \emph{Some remarks on the regularity of
anisotropic variational problems}\textit{,} Rend. Accad. Naz. Sci. XV Mem.
Mat. \textbf{17} (1993), 229-239.

\bibitem[T]{Ta1}G. Talenti, \emph{Elliptic equations and rearrangements}, Ann.
Sc. Norm. Sup. Pisa IV \textbf{3} (1976), 697-718.
\end{thebibliography}
\end{document}